\documentclass[12pt]{amsart}

\usepackage{amsmath}
\usepackage[curve]{xypic}
\usepackage{caption}
\usepackage{xspace}
\usepackage{cite}
\usepackage{enumitem}
\usepackage{color}
\usepackage{url}
\usepackage[usenames,dvipsnames]{xcolor}
\usepackage{graphicx}

\makeatletter % makes "@" a normal character so that it can be used in macros
\def\@cite#1#2{{\m@th\upshape\bfseries%
[{#1\if@tempswa{\m@th\upshape\mdseries, #2}\fi}]}}
\makeatother % returns "@" to its normal state

\theoremstyle{plain}
\newtheorem{thm}{Theorem}[section]% subsection

\newtheorem{prop}[thm]{Proposition}
\newtheorem{lem}[thm]{Lemma}
\theoremstyle{definition}

\newtheorem{prob}[thm]{Problem}

\theoremstyle{remark}
\newtheorem{rem}[thm]{Remark}

\numberwithin{equation}{subsection}
\renewcommand{\bold}[1]{\medskip \noindent {\bf #1 }\nopagebreak}

\newcommand{\nc}{\newcommand}
\newcommand{\rnc}{\renewcommand}
\newcommand{\e}{\varepsilon}

\newcommand{\eb}[1]{\emph{\textbf{#1}}}
\nc\bA{\mathbb{A}}
\nc\bB{\mathbb{B}}
\nc\bC{\mathbb{C}}
\nc\bD{\mathbb{D}}
\nc\bE{\mathbb{E}}
\nc\bF{\mathbb{F}}
\nc\bG{\mathbb{G}}
\nc\bH{\mathbb{H}}
\nc\bI{\mathbb{I}}
\nc{\bJ}{\mathbb{J}}
\nc\bK{\mathbb{K}}
\nc\bL{\mathbb{L}}
\nc\bM{\mathbb{M}}
\nc\bN{\mathbb{N}}
\nc\bO{\mathbb{O}}
\nc\bP{\mathbb{P}}
\nc\bQ{\mathbb{Q}}
\nc\bR{\mathbb{R}}
\nc\bS{\mathbb{S}}
\nc\bT{\mathbb{T}}
\nc\bU{\mathbb{U}}
\nc\bV{\mathbb{V}}
\nc\bW{\mathbb{W}}
\nc\bY{\mathbb{Y}}
\nc\bX{\mathbb{X}}
\nc\bZ{\mathbb{Z}}
\nc\cA{\mathcal{A}}
\nc\cB{\mathcal{B}}
\nc\cC{\mathcal{C}}
\rnc\cD{\mathcal{D}}
\nc\cE{\mathcal{E}}
\nc\cF{\mathcal{F}}
\nc\cG{\mathcal{G}}
\rnc\cH{\mathcal{H}}
\nc\cI{\mathcal{I}}
\nc{\cJ}{\mathcal{J}}
\nc\cK{\mathcal{K}}
\rnc\cL{\mathcal{L}}
\nc\cM{\mathcal{M}}
\nc\cN{\mathcal{N}}
\nc\cO{\mathcal{O}}
\nc\cP{\mathcal{P}}
\nc\cQ{\mathcal{Q}}
\rnc\cR{\mathcal{R}}
\nc\cS{\mathcal{S}}
\nc\cT{\mathcal{T}}
\nc\cU{\mathcal{U}}
\nc\cV{\mathcal{V}}
\nc\cW{\mathcal{W}}
\nc\cY{\mathcal{Y}}
\nc\cX{\mathcal{X}}
\nc\cZ{\mathcal{Z}}

\nc\Gp{\Gamma_{\negmedspace\perp}}
\nc\oGp{\overline{\Gamma}_{\negmedspace\perp}}

\nc{\dmo}{\DeclareMathOperator}
\rnc{\Re}{\operatorname{Re}}
\rnc{\Im}{\operatorname{Im}}
\nc{\dM}{\overline{\cM}}
\nc{\spn}{\operatorname{span}} %DO NOT USE "span." Renewing "span" breaks align
\dmo{\rank}{rank}
\dmo{\End}{End}
\dmo{\Jac}{Jac}
\dmo{\Id}{Id}
\dmo{\lcm}{lcm}
\dmo{\Aff}{Aff}
\dmo{\tr}{tr}
\dmo{\Aut}{Aut}

\nc{\good}{Hodge-Teichm\"uller\xspace}
\nc{\newh}{h}

\begin{document}
%%%%%%%%%%%%%%%%%%%%%%%%%%%%%%%%%%%%%%%%%%%

\title[Hodge-Teichm\"uller planes]{Hodge-Teichm\"uller planes and finiteness results for Teichm\"uller curves}
%\thanks{This research was supported in part by NSERC.\\ The author thanks...}
%
\author[C.Matheus]{Carlos~Matheus}
\address{Universit\'e Paris 13, Sorbonne Paris Cit\'e, LAGA, CNRS (UMR 7539), F-93430, Villetaneuse, France}
\email{matheus.cmss@gmail.com}
\author[A.Wright]{Alex~Wright}
\address{Math\ Department\\University of Chicago\\
5734 South University Avenue\\
Chicago, IL 60637}
\email{alexmwright@gmail.com}
%
%\subjclass[2010]{22E60, 15A57, 17B20, 58C35}
%\keywords{}
\date{July 4, 2014.}
%\dedicatory{Preliminary version. Comments welcome.}

\begin{abstract}
We prove that there are only finitely many algebraically primitive Teichm\"uller curves in the minimal stratum in each prime genus at least 3. The proof is based on the study of certain special planes in the first cohomology of a translation surface which we call \good planes.

We also show that algebraically primitive Teichm\"uller curves are not dense in any connected component of any stratum in genus at least 3; the closure of the union of all such curves (in a fixed stratum) is equal to a finite union of affine invariant submanifolds with unlikely properties. Results of this type hold even without the assumption of algebraic primitivity.

Combined with work of Nguyen and the second author, a corollary of our results is that there are at most finitely many non-arithmetic Teichm\"uller curves in $\cH(4)^{hyp}$.
%We prove that there are only finitely many algebraically primitive Teichm\"uller curves in the minimal stratum in each prime genus at least 3.  This complements previous results by Calta and McMullen (independently), McMullen, Bainbridge-M\"oller, and M\"oller (resp.) on the number of algebraically primitive Teichm\"uller curves in $\mathcal{H}(2)$, $\mathcal{H}(1,1)$,  $\mathcal{H}(3,1)$, and  $\mathcal{H}(g-1,g-1)^{hyp}$ (resp.).
%
%The proof is based on the study of certain special planes in the first cohomology of a translation surface which we call \good planes. The consideration of these planes is motivated by M\"oller's description of the VHS over a Teichm\"uller curve.
%
%We also show that algebraically primitive Teichm\"uller curves are not dense in any connected component of any stratum in genus at least 3; the closure of the union of all such curves (in a fixed stratum) is equal to a finite union of affine invariant submanifolds with unlikely properties. Results of this type hold even without the assumption of algebraic primitivity.
%
%Combined with work of Nguyen and the second author, a corollary of these results is that there are at most finitely many non-arithmetic Teichm\"uller curves in $\cH(4)^{hyp}$.
\end{abstract}

\maketitle
% removes page number from first page
\thispagestyle{empty}

%%%%%%%%%%%%%%%%%%%
% TABLE OF CONTENTS
%%%%%%%%%%%%%%%%%%%
 %allows subsections (depth 1) to be displayed in table of contents
%\newpage
%\setcounter{tocdepth}{2}
%\tableofcontents
%\newpage

%%%%%%%%%%%%%%%%%%%
%%%%%%%%%%%%%%%%%%%
% INTRODUCTION
%%%%%%%%%%%%%%%%%%%
%%%%%%%%%%%%%%%%%%%

\section{Introduction}\label{s.introduction}

%%%%%%%%%%%%%%%%%%%
%%%%%%%%%%%%%%%%%%%

\bold{Context.} There is an analogy between the dynamics of the natural $SL(2,\mathbb{R})$--action on the moduli space of unit area translation surfaces and homogenous space dynamics. This analogy has been very fruitful in the study of rational billiards, interval exchange transformations,  and translation flows, as well as related problems in physics.  From this point of view, a key goal is to obtain a classification of $SL(2,\mathbb{R})$--orbit closures, in partial analogy with Ratner's Theorem for homogeneous spaces.
%surfaces , such as counting closed trajectories, on translation surfaces and rational billiards \cite{V4, EMa, EM, EMM}.

Quite a lot of progress has occurred on this problem. McMullen has classified orbit closures in genus 2 \cite{Mc5}, and Eskin-Mirzakhani-Mohammadi have proved that all orbit closures are \emph{affine invariant submanifolds} \cite{EM, EMM}. Progress on understanding affine invariant submanifolds is ongoing \cite{W4, W3}. Despite all this, a classification even just of closed orbits is not available in any genus greater than 2.

As is common practice, we will call closed $SL(2,\mathbb{R})$--orbits \emph{Teichm\"uller curves}. (The projections of these orbits to the moduli space of Riemann surfaces are also called Teichm\"uller curves; these projections are in fact complex algebraic curves which are isometrically immersed with respect to the Teichm\"uller metric.) Teichm\"uller curves yield examples of translation surfaces and billiards with optimal dynamical properties \cite{V}, as well as families of algebraic curves with extraordinary algebro-geometric properties \cite{M, M2, Mc3}.

As is the case with lattices in Lie groups, Teichm\"uller curves come in commensurability classes, and are either arithmetic or not. In the non-arithmetic case, every commensurability class contains a unique minimal representative, which is called primitive \cite{M2}. Teichm\"uller curves that are primitive for algebraic reasons are called algebraically primitive, and it is precisely these curves which are ``as non-arithmetic as possible."

There are infinitely many algebraically primitive Teichm\"uller curves in $\cH(2)$, which were constructed independently by Calta and McMullen \cite{Ca, Mc} and classified by McMullen \cite{McM:spin}. McMullen showed that there is only one algebraically primitive Teichm\"uller curve in $\cH(1,1)$ \cite{Mc4}. Finiteness of algebraically primitive Teichm\"uller curves is known in $\cH(g-1,g-1)^{hyp}$ by work of M\"oller \cite{M3}, and in $\cH(3,1)$ by work of Bainbridge-M\"oller \cite{BaM}. Bainbridge and M\"oller have informed the second author that together with Habegger they have recently established new finiteness results that are complementary to those in this paper. Their methods are very different from ours.% in all strata in genus 3 and in the principal stratum in all genus.

%\begin{rem} In some sense, the methods in this paper complement those used by Bainbridge, M\"oller and Habegger:  the latter are better suited to  non-minimal strata,  and the former are currently better suited to the case of minimal strata.
%\end{rem}

The assumption of \emph{algebraic} primitivity is made primarily for two reasons. Firstly, the key tools include  M\"oller's algebro-geometric results \cite{M, M2}, and these results are most restrictive in the algebraically primitive case.  Secondly, in some situations finiteness is not true for primitive but not algebraically primitive curves. For example, McMullen \cite{Mc2} has constructed infinitely many primitive Teichm\"uller curves in the minimal stratum in genus 3 and 4 (none of which are algebraically primitive).

%%%%%%%%%%%%%%%%%%%
%%%%%%%%%%%%%%%%%%%

\bold{Statement of results.}  Our main result is

\begin{thm}\label{T:main}
There are only finitely many algebraically primitive Teichm\"uller curves in the minimal stratum in prime genus at least 3. Furthermore, algebraically primitive Teichm\"uller curves are not dense in any connected component of any stratum in genus at least 3.
\end{thm}

A purely algebro-geometric reformulation of the first statement of Theorem \ref{T:main} is possible: the locus of all eigenforms in the projectivized minimal stratum in prime genus at least 3 is a countable set together with at most a finite number of algebraic curves.\footnote{{Teichm\"uller curves can be characterized by the torsion and real multiplication conditions discovered by M\"oller \cite{M, M2}. In particular, any curve in the projectivized minimal stratum, all of whose points $(X, \bC\cdot\omega)$ have the property that $\Jac(X)$ has real multiplication with $\omega$ as an eigenform, must be the projectivization of  a Teichm\"uller curve. Added in proof: A generalization of this unpublished observation of the second author has been very recently used by Filip, in his work showing that all $SL(2,\bR)$ orbit closures are varieties \cite{Fi, Fi2}.}} %We defer further discussion of this reformulation to a future paper, as it relies on currently unpublished work of the second author.

 A key new idea in the proof of Theorem \ref{T:main} is the study of \good planes. A \emph{\good plane} at a translation surface $M$ is a plane $P\subset H^1(M,\bR)$ such that for any $h\in SL(2,\bR)$, the plane\footnote{See Section \ref{S:HTdefn} for details.} $hP\subset H^1(hM, \bR)$ satisfies
$\dim_\bC (hP)^{1,0}=1$. It is then automatic that $\dim_\bC (hP)^{0,1}=1$. Here
$$(hP)^{1,0} = H^{1,0}(hM)\cap (hP\otimes \bC) \quad\text{and}\quad (hP)^{0,1} = H^{0,1}(hM)\cap (hP\otimes \bC).$$
\good planes are thusly named because they respect the Hodge decomposition along the entire Teichm\"uller disk of a translation surface. (The Teichm\"uller disk of a translation surface is the projection of the $SL(2,\bR)$--orbit to the moduli space of Riemann surfaces.) The consideration of these planes is motivated by M\"oller's description of the variation of Hodge structure over a Teichm\"uller curve \cite{M}.

\begin{thm}\label{T:manygood}
Suppose $\cM$ is an affine invariant submanifold that contains a dense set of Teichm\"uller curves whose trace fields have degree at least $k$. Then every translation surface in $\cM$ has $k$ orthogonal \good planes.

In particular, suppose $\cM$ is an affine invariant submanifold of genus $g$ translation surfaces, and algebraically primitive Teichm\"uller curves are dense in $\cM$. Then every translation surface in $\cM$ has $g$ orthogonal \good planes.
\end{thm}

The orthogonality is either with respect to the Hodge inner product, or the symplectic form: for a \good plane, the orthogonal complement is the same for both. When we say that a translation surface ``has $k$ orthogonal \good planes" more precisely we mean that the translation surface has a set of $k$ pairwise orthogonal \good planes (there may also be more).

If $M=(X,\omega)$, then the tautological plane $P=\spn_\bR(\Re(\omega), \Im(\omega))$ is \good. In genus two, the orthogonal complement of this plane always gives a second \good plane. However,

\begin{thm}\label{T:notmanygood}
In any connected component of any stratum of unit area translation surfaces in genus $g\geq 3$, there is a translation surface which does not have $g-1$ orthogonal \good planes.
\end{thm}

The final ingredients in the proof of Theorem \ref{T:main} are

\begin{thm}[Eskin-Mirzakhani-Mohammadi \cite{EMM}]\label{T:EMM}
Any closed $SL(2,\bR)$--invariant subset of a stratum is a finite union of affine invariant submanifolds.
\end{thm}

\begin{thm}[Wright \cite{W3}]\label{T:dense_in_stratum}
Let $m\geq 2$ be prime. Any affine invariant submanifold of the minimal stratum in genus $m$ which properly contains an algebraically primitive Teichm\"uller curve must be equal to a connected component of the stratum.
\end{thm}

\begin{proof}[\eb{Proof of Theorem \ref{T:main} using Theorems \ref{T:manygood}, \ref{T:notmanygood}, \ref{T:EMM} and \ref{T:dense_in_stratum}.}]
Let $\cH$ be a connected component of a stratum of genus $g\geq 3$ unit area translation surfaces. By Theorem \ref{T:EMM}, the closure of the union of all algebraically primitive Teichm\"uller curves in $\cH$ is a finite union of affine invariant submanifolds $\cM_{i}, i=1,\ldots,\ell$.

By Theorem \ref{T:manygood}, every translation surface in each $\cM_i$ must have $g$ orthogonal \good planes. Thus Theorem \ref{T:notmanygood} gives that no $\cM_i$ can be equal to $\cH$. In particular, algebraically primitive Teichm\"uller curves are not dense in $\cH$.

Now suppose $\cH$ is a connected component of the minimal stratum in prime genus. Then Theorem \ref{T:dense_in_stratum} gives that each $\cM_i$ must be a Teichm\"uller curve. In particular, there are only finitely many algebraically primitive Teichm\"uller curves in $\cH$.
\end{proof}

\bold{Additional results.}%\ann{{\color{red}The conditional theorem has been removed.}}
In recent work, Nguyen and the second author have shown that there are no affine invariant submanifolds in $\cH(4)^{hyp}$, except Teichm\"uller curves and $\cH(4)^{hyp}$ itself \cite{NW}. From this we get

\begin{thm}[Matheus-Nguyen-Wright]\label{th:fin}
There are at most finitely many non-arithmetic Teichm\"uller curves in $\cH(4)^{hyp}$.
\end{thm}

Theorem \ref{th:fin} is false in $\cH(4)^{\rm odd}$ by work of McMullen \cite{Mc2}. There are at present two known non-arithmetic Teichm\"uller curves in $\cH(4)^{\rm hyp}$, corresponding to the regular $7$-gon and $12$-gon. These examples are due to Veech \cite{V}. The first is algebraically primitive, and the second is not. Bainbridge and M\"oller  have conjectured that the regular $7$-gon is the only algebraically primitive Teichm\"uller curve in $\cH(4)^{\rm hyp}$ \cite[Ex. 14.4]{BaM}.

\begin{proof}[\eb{Proof using Theorems \ref{T:manygood}, \ref{T:notmanygood}, and \ref{T:EMM}.}]
By Theorem \ref{T:EMM}, the closure of the union of all non-arithmetic  Teichm\"uller curves in $\cH(4)^{hyp}$ is a finite union of affine invariant submanifolds $\cM_{i}, i=1,\ldots,\ell$. As above, it suffices to rule out the case that some $\cM_i$ properly contains a Teichm\"uller curve.

If this were the case, the work of Nguyen-Wright \cite{NW} would give that $\cM_i=\cH(4)^{hyp}$. Theorem \ref{T:manygood} would thus give that each translation surface in $\cH(4)^{hyp}$ has 2 orthogonal \good planes, contradicting Theorem \ref{T:notmanygood}.
\end{proof}

Added in proof: recently a similar argument has been used to give results in the other component of $\cH(4)$ \cite{ANW}.

%%%%%%%%%%%%%%%%%%%
%%%%%%%%%%%%%%%%%%%

\bold{More on the proofs.} We will see that the work of M\"oller \cite{M} shows that any translation surface on a Teichm\"uller curve has at least $k$ orthogonal \good planes, where $k$ is the degree of the trace field. Theorem \ref{T:manygood} will then follow from the observation that the limit of \good planes is a \good plane.

The proof of Theorem \ref{T:notmanygood} will be an induction on the genus and number of zeros. The base case is

\begin{thm}\label{T:basecase}
In both connected components $\cH(4)^{hyp}$ and $\cH(4)^{odd}$ of the minimal stratum in genus 3, there are translation surfaces that do not have 2 orthogonal \good planes.
\end{thm}

Given a translation surface $M$, we will consider the action of its affine group on $H^1(M,\bR)$. We will show that the Zariski closure of the image of this action preserves the set of \good planes. Theorem \ref{T:basecase} will then be established by exhibiting specific square-tiled surfaces $M$ in $\cH(4)^{hyp}$ and $\cH(4)^{odd}$ for which the Zariski closure is large enough to imply that there are no \good planes besides the tautological plane.

The inductive step of the proof of Theorem \ref{T:notmanygood} is

\begin{thm}\label{T:induct}
Let $\cH$ be a connected component of a stratum of genus $g\geq 3$ translation surfaces. Suppose  that every translation surface in $\cH$ has $k$ orthogonal  \good planes.

Then there is a connected component $\cH'$ of the minimal stratum of genus $g$ translation surfaces in which every translation surface has at least $k$ orthogonal \good planes.

If $\cH$ is already a connected component of the minimal stratum $\cH(2g-2)$, then there is a connected component $\cH'$ of $\cH(2g-4)$ in which every translation surface has at least $k-1$ orthogonal \good planes.
\end{thm}

The proof of Theorem \ref{T:induct} uses results developed by Kontsevich-Zorich for the classification of connected components of strata \cite{KZ}. Specifically, we will consider degenerations of translation surfaces $M_n\to M$ that arise from opening up a zero on $M$ or bubbling off a handle on $M$, and we will show that the limit of \good planes on $M_n$ is a \good plane on $M$.

\begin{proof}[\eb{Proof of Theorem \ref{T:notmanygood} using Theorems \ref{T:basecase} and \ref{T:induct}.}]
Suppose, in order to find a contradiction, that there is a connected component of a stratum of genus $g \geq 3$ translation surfaces in which every translation surface has $g-1$ orthogonal \good planes. By applying Theorem \ref{T:induct} a number of times, we eventually see that a connected component of $\cH(4)$ must consist entirely of translation surface with $2$ orthogonal \good planes. However, this contradicts Theorem \ref{T:basecase}.
\end{proof}

%%%%%%%%%%%%%%%%%%%
%%%%%%%%%%%%%%%%%%%

\bold{References.} For a flat geometry perspective on translation surfaces, see for example \cite{GJ}, \cite{HS2}, and \cite{SW}.

There are only very few examples of primitive Teichm\"uller curves known. They are the Prym curves in genus 2, 3 and 4 \cite{Mc, Mc4} (see also \cite{Ca} for the genus 2 case); the Veech-Ward-Bouw-M\"oller curves \cite{BM}; and two sporadic examples, one due to Vorobets in $\cH(6)$, and another due to Kenyon-Smillie in $\cH(1,3)$ \cite{HS, KS}. These sporadic examples correspond to billiards in the  $(\pi/5, \pi/3, 7\pi/15)$ and $(2\pi/9, \pi/3, 4\pi/9)$ triangles respectively, and both Teichm\"uller curves are algebraically primitive.

See \cite{W2} for an introduction to the Veech-Ward-Bouw-M\"oller curves and more background on Teichm\"uller curves, and Hooper for a flat geometry perspective \cite{H}.

The (primitive) Teichm\"uller curves in $\cH(2)$ have been extensively studied by McMullen \cite{McM:spin}, Bainbridge \cite{Ba}, and Mukamel \cite{Mu:orb, Mu:alg}. The Prym curves are not quite as well understood but have also been extensively studied by M\"oller \cite{Mo:prym} and Lanneau-Nguyen \cite{LN:prym}.

McMullen has shown that Teichm\"uller curves in moduli space are rigid \cite{Mc3}.

%%%%%%%%%%%%%%%%%%%
%%%%%%%%%%%%%%%%%%%

\bold{Organization.} In Section \ref{s.definitions}, we recall some basic material on translation surfaces and their moduli spaces. In particular, the notions of Teichm\"uller curves, affine invariant submanifolds and the Hodge inner product are quickly reviewed. In this section we give the precise definition of Hodge-Teichm\"uller planes.

In Section \ref{s.HTplanes}, we prove Theorem \ref{T:main}. In Section \ref{s.4}, we prove Theorem \ref{T:basecase}. In Section \ref{s.induct} we prove Theorem \ref{T:induct} modulo some technical results, which are proved in Section \ref{S:conv}.

We conclude the paper in Section \ref{S:open} by listing a few open problems that may now be accessible.
% by combining M\"oller's work on the variations of Hodge structures (VHS) over Teichm\"uller curves with some elementary properties of \good planes under limits.
%by constructing explicit square-tiled surfaces in $\cH(4)$ with no non-tautological \good planes.

%%%%%%%%%%%%%%%%%%%
%%%%%%%%%%%%%%%%%%%

\bold{Acknowledgements.} The authors are grateful to Pascal Hubert, Erwan Lanneau, Jean-Christophe Yoccoz and Anton Zorich for organizing the excellent summer school ``Algebraic Geometry'' (from September 24 to 28, 2012) in Roscoff, France, where the authors began their collaboration. The authors  also thank Matt Bainbridge, Alex Eskin, Howard Masur, and Anton Zorich for useful conversations, and  Matt Bainbridge, Curtis McMullen for helpful comments on an earlier version of this paper.

We thank Anton Zorich for permission to reproduce figures \ref{F:split} and \ref{F:handle} from \cite{EMZ} and \cite{KZ}.

We thank the two referees for very helpful comments resulting in significant improvements to this paper. We thank Ronen Mukamel and Paul Apisa for helpful comments on Section 2.3. 

The first author was partially supported by the French ANR grant ``GeoDyM'' (ANR-11-BS01-0004) and by the Balzan Research Project of
J. Palis.

%%%%%%%%%%%%%%%%%%%
%%%%%%%%%%%%%%%%%%%
% DEFINITIONS
%%%%%%%%%%%%%%%%%%%
%%%%%%%%%%%%%%%%%%%

\section{Definitions}\label{s.definitions}

In this section we define all the terms used in the introduction, and provide some basic background. Additional background material will be recalled as necessary in the subsequent sections.

%%%%%%%%%%%%%%%%%%%
%%%%%%%%%%%%%%%%%%%

\subsection{Abelian differentials and translation surfaces.} Let $X$ be a Riemann surface of genus $g\geq 1$. A (non-trivial) Abelian differential $\omega$ on $X$ is simply a (non-zero) holomorphic $1$-form. Given an Abelian differential $\omega$ on a Riemann surface $X$, denote by $\Sigma$ the finite set of zeroes of $\omega$. By locally integrating $\omega$ outside $\Sigma$, we obtain a translation atlas on $X-\Sigma$, i.e., an atlas whose change of coordinates are translations of the plane $\mathbb{R}^2$. From the point of view of the translation atlas, a zero $p$ of $\omega$ is a conical singularity with total angle $2\pi (k+1)$ where $k$ is the order of zero $p$. Conversely, given a translation atlas with conical singularities as above, we recover an Abelian differential by pulling back the form $dz=dx+i dy$ on $\mathbb{R}^2$ via the charts of the translation atlas. In other words, there is a one-to-one correspondence between Abelian differentials and translation surfaces.

Note that $SL(2,\mathbb{R})$ acts on the set of translation surfaces by post-composition with the charts of translation atlas. That is, $SL(2,\mathbb{R})$ acts naturally on Abelian differentials. Moreover, this $SL(2,\mathbb{R})$-action on Abelian differential preserves the total area function $\textrm{Area}(\omega):=\frac{i}{2}\int_X \omega\wedge\overline{\omega}$.

We will typically denote a translation surface by $M=(X,\omega)$.

%%%%%%%%%%%%%%%%%%%
%%%%%%%%%%%%%%%%%%%
% of real dimension $4g+2\sigma-3$

\subsection{Moduli spaces of Abelian differentials}  Let $\cH_g$ denote the moduli space of Abelian differentials $\omega$ with unit total area (i.e., $\textrm{Area}(\omega)=1$) on Riemann surfaces $X$ of genus $g\geq 1$. It is well known that $\cH_g$ is naturally stratified, where the strata $\cH(k_1,\dots,k_{\sigma})$ are obtained by collecting together Abelian differentials $(X,\omega)\in\cH_g$ whose orders of zeros is prescribed by the list $(k_1,\dots,k_{\sigma})$. The action of $SL(2,\mathbb{R})$ on $\cH_g$ preserves each stratum $\cH(k_1,\dots,k_{\sigma})$. For later use (in this section only), let us denote by $$\mathbb{R}\cH(k_1,\dots, k_{\sigma})=\{(M,t\omega): t\in\mathbb{R}, (M,\omega)\in\cH(k_1,\dots,k_{\sigma})\}.$$

Kontsevich and Zorich have classified the connected components of strata \cite{KZ}. In particular, each stratum has at most 3 connected components.

Work of Masur and Veech shows that the $SL(2,\bR)$--action on any connected component of a stratum is ergodic with respect to certain fully supported invariant probabilities called Masur-Veech measures \cite{Ma2, V2}. In particular, typical translation surfaces (with respect to Masur-Veech measures) have dense orbits.

%%%%%%%%%%%%%%%%%%%
%%%%%%%%%%%%%%%%%%%

\subsection{\good planes.}\label{S:HTdefn} We now give a precise definition of a \good plane. 

Let $\cT_g$ denote the Teichm\"uller space of genus $g$ Riemann surfaces. A point in $\cT_g$ is an isomorphism class of Riemann surface $X$ together with a marking $f: S\to X$, which is an isotopy class of homeomorphisms from a fixed topological surface $S$ to $X$. Let $H^{1,0}$ denote the bundle of unit area Abelian differentials over $\cT_g$. The fiber of $H^{1,0}$ over a marked Riemann surface $X$ consists of unit area Abelian differentials in $H^{1,0}(X)$, where $H^{1,0}(X)$ denotes the vector space of Abelian differentials on $X$. There is a map $H^{1,0}$ to $\cH_g$ which simply forgets the marking on the Riemann surface. Put differently, $\cH_g$ is the quotient of $H^{1,0}$ by the mapping class group. 

Let $H^1_{\mathbb{R}}$ be the bundle over $H^{1,0}$ whose fiber over a marked Riemann surface $X$ with Abelian differential $\omega$ is $H^1(X,\mathbb{R})$. Since $H^1(X,\mathbb{R})$ is canonically identified with $H^1(S,\mathbb{R})$ via the marking, this bundle is a trivial bundle. 

There is a natural $SL(2,\bR)$ action on the bundle $H^{1,0}$ over Teichm\"uller space, which projects to the $SL(2,\bR)$ action on $\cH_g$. There is also an action of $SL(2,\bR)$ on the bundle $H^1_\bR$, which is typically used to define the Kontzevich-Zorich cocycle (see the next remark). This action is defined to act trivially on the fibers. In other words, given a cohomology class $\eta$ on a marked $(X,\omega)$, and given $h\in SL(2,\bR)$, the cohomology class $h(\eta)\in H^1(h(X,\omega), \bR)$ is defined to be the image of $\eta$ under the isomorphism of cohomology groups provided by affine map induced by $h$. (Since the marking on $h(X,\omega)$ is defined by post-composition with this affine map, $h(\eta)$ is identified to the same cohomology class in $H^1(S,\bR)$ as $\eta$, which is why the action is said to be trivial on fibers.) 

Consider a plane $P\subset H^1(X,\bR)$, where $(X,\omega)$ is an (unmarked) translation surface. If we pick a marking $S\to X$, we may view $P$ as a plane in a fiber of the bundle $H^1_\bR$, and hence act on it by $SL(2,\bR)$.  Given a plane $P\subset H^1(X,\mathbb{R})$, we denote by $P^{1,0}=H^{1,0}(X)\cap P_{\mathbb{C}}$ and $P^{0,1}=H^{0,1}(X)\cap P_{\mathbb{C}}$ where $P_{\mathbb{C}}=P\otimes\mathbb{C}\subset H^1(X,\mathbb{C})$ is the complexification of $P$.

We define $P$ to be \good if for any (equivalently, all) choices of marking $S\to X$, and any $h\in SL(2,\bR)$, we have that $(hP)^{1,0}$ and $(hP)^{0,1}$ are both one complex dimensional. (In fact the second condition is guaranteed by the first: If $(hP)^{1,0}$ is one dimensional, then so is the complex conjugate subspace $(hP)^{0,1}$.)

%Let $\widetilde{SL(2,\bR)}$ denote the universal cover of $SL(2,\bR)$. Each $\newh \in \widetilde{SL(2,\bR)}$ may be considered a map  $$\newh :[0,1] \to SL(2,\bR)$$ with $\newh (0)=e$, up to isotopy fixing the endpoints. We will denote the second endpoint as $g=\newh (1)$.

%If $M$ is a translation surface, $\newh $ determines a one parameter family of translation surfaces $\newh (t)M$. Topologically, this family is $[0,1]\times M$, and thus the cohomology of the fibers over 0 and 1 are identified via a isomorphism $$\newh _*:H^1(M)\to H^1(gM).$$
%This is the map obtained by dragging cohomology classes along the family $\newh (t)M$ from $t=0$ to $t=1$.

%A \emph{\good plane} at a translation surface $M$ is a plane $P\subset H^1(M,\bR)$ such that for all $\newh \in \widetilde{SL(2,\bR)}$, we have that $\newh _*P\subset H^1(gM)$ satisfies $\dim_\bC (\newh _*P)^{1,0}=1$.

%Here (and in what follows), given a plane $P\subset H^1(M,\mathbb{R})$, we denote by $P^{1,0}=H^{1,0}(M)\cap P_{\mathbb{C}}$ and $P^{0,1}=H^{0,1}(M)\cap P_{\mathbb{C}}$ where $P_{\mathbb{C}}=P\otimes\mathbb{C}\subset H^1(M,\mathbb{C})$ is the complexification of $P$.

The tautological plane at $M=(X,\omega)$ is defined to be $$\spn_\bR(\Re(\omega), \Im(\omega))\subset H^1(M,\bR).$$ It is a \good plane.

\begin{rem}\label{R:annoying1}
The action of $SL(2,\bR)$ on $H^1_\bR$ commutes with the action of the mapping class group, so induces an action on the quotient of $H^1_\bR$ by the mapping class group. This action is sometimes referred to as the Kontsevich-Zorich cocycle, although strictly speaking it is not a cocycle in the usual dynamical terms (unless it is considered to be only measurably defined, or the definitions are modified slightly), since the quotient of $H^1_\bR$ by the mapping class group is not strictly speaking a vector bundle in the non-orbifold sense; indeed, over points $(X,\omega)$ where $\Aut(X, \omega)$ is non-trivial, the fiber is $H^1(X, \bR)/\Aut(X,\omega)$, and this is not always a vector space. See \ref{ss.4.def} for a definition of $\Aut(X,\omega)$.
\end{rem}

\begin{rem}\label{R:annoying2}
It would also be possible to define \good planes using the orbifold bundle over moduli space obtained as the quotient of the bundle $H^1_\bR$ over Teichm\"uller space by the action of the mapping class group. Because of the issues indicated in previous remark we found it simpler to phrase the definition in Teichm\"uller space. 

Instead of using Teichm\"uller space, we could use the moduli spaces of translation surfaces with a level structure; these moduli spaces are fine moduli spaces and are smooth manifolds instead of orbifolds. Over them, the bundle $H^1_\bR$ is a typical (non-orbifold) bundle with an action of $SL(2,\bR)$. 

These alternate approaches are not substantively different from the approach we have followed.
\end{rem}

%%%%%%%%%%%%%%%%%%%
%%%%%%%%%%%%%%%%%%%

\subsection{Teichm\"uller curves.}

We will call closed (as subsets of the stratum) $SL(2,\mathbb{R})$--orbits \emph{Teichm\"uller curves}. The projections of these orbits to the moduli space of Riemann surfaces are also called Teichm\"uller curves; these projections are in fact algebraic curves which are isometrically immersed with respect to the Teichm\"uller metric.

Given such a curve $\cC$ in moduli space, at each $X\in \cC$ consider the initial quadratic differential $q$ of any Teichm\"uller geodesic segment in  $\cC$ starting at $X$. If $q=\omega^2$ is the square of an Abelian differential, then in fact the $SL(2,\mathbb{R})$--orbit of $(X,\omega)$ is closed, and its projection to moduli space is exactly $\cC$. If $q$ is not the square of an Abelian differential, a double cover construction nonetheless yields a corresponding closed orbit in a moduli space of higher genus translation surfaces. Thus the two perspectives, closed orbits and curves in moduli space, are equivalent.

Denote by $SL(X,\omega)\subset SL(2,\bR)$ the stabilizer of the translation surface $(X,\omega)$. If the $SL(2,\bR)$--orbit of $M=(X,\omega)$ is closed, then Smillie's Theorem asserts that $SL(X,\omega)$ is a lattice in $SL(2,\mathbb{R})$ \cite{V5, SW2}. In this case, the $SL(2,\mathbb{R})$-orbit of $(X,\omega)$ is isomorphic to $SL(2,\mathbb{R})/SL(X,\omega)$, that is, the unit cotangent bundle of a complete non-compact hyperbolic surface $\mathbb{H}/SL(X,\omega)$ immersed in  moduli space.

The field $\bQ[\tr(h): h\in SL(X,\omega)]$ is called the trace field of $(X,\omega)$. If $(X,\omega)$ lies on a Teichm\"uller curve, this field is also called the trace field of this Teichm\"uller curve. It is well known that the trace field of a genus $g$ translation surface is a number field of degree at most $g$ \cite{KS, Mc}.

If a Teichm\"uller curve has trace field equal to $\bQ$, then it is generated by a square-tiled surface \cite{Mc6, GJ}. In this case we say the Teichm\"uller curve is \emph{arithmetic}. A square-tiled surface is by definition a translation covering of the torus branched over one point (see below).

%%%%%%%%%%%%%%%%%%%
%%%%%%%%%%%%%%%%%%%

\subsection{Commensurability and primitivity.} A \emph{translation covering} from a translation surface $(X,\omega)$ to another translation surface $(X', \omega')$ is defined to be a branched covering of Riemann surfaces $f:X\to X'$ such that $f^*(\omega')=\omega$. A translation surface is called \emph{primitive} if it does not translation cover anything of smaller genus.

The following result was established by M\"oller \cite{M2}, and is also explained in \cite{Mc2}.

\begin{thm}[M\"oller]\label{T:prim}
Every translation surface $M$ is a translation covering of a primitive translation surface $M_{prim}$. If $M_{prim}$ has genus greater than 1, then it is unique. If $M$ lies on a Teichm\"uller curve, then so does $M_{prim}$, and in this case $SL(M)$ and $SL(M_{prim})$ are commensurable lattices in $SL(2,\bR)$.
\end{thm}

When a Teichm\"uller curve consists of primitive translation surfaces, we say this Teichm\"uller curve is primitive. When two Teichm\"uller curves $\cC_1$ and $\cC_2$ have a pair of translation surfaces $M_1\in \cC_1$ and $M_2\in \cC_2$ covering the same primitive translation surface, then we say that $\cC_1$ and $\cC_2$ are commensurable. Theorem \ref{T:prim} gives that any commensurability class of non-arithmetic Teichm\"uller curves contains a unique primitive Teichm\"uller curve.

Recall that the trace field of a genus $g$ translation surface is a number field of degree at most $g$. Recall also that the trace field is a commensurability invariant \cite{KS}. It follows that if the degree of the trace field of a Teichm\"uller curve of genus $g$ translation surfaces is $g$, then this Teichm\"uller curve is primitive. In this case we say that the Teichm\"uller curve is \emph{algebraically primitive}.

We think of the degree of the trace field as a measure of how non-arithmetic a Teichm\"uller curve is, and thus algebraically primitive curves are as far as possible from being arithmetic.

%%%%%%%%%%%%%%%%%%%
%%%%%%%%%%%%%%%%%%%

\subsection{Affine invariant submanifolds.} %This subsection is copied and pasted from another paper.
Let $\cH$ be a finite cover of a stratum $\cH(k_1,\dots,k_{\sigma})$ which is a a manifold instead of an orbifold.\footnote{This guarantees that there is a family of Riemann surfaces over $\cH$, such that fiber over a point $\cH$ is a Riemann surface isomorphic to this point. Over such $\cH$ there is a well defined (non-orbifold) flat vector bundle whose fiber over a point is its first real cohomology. $\cH$ may be taken to be the quotient of Teichm\"uller space by any torsion free finite index subgroup of the mapping class group.} For example, $\cH$ may be taken to be the moduli space of translation surfaces in some given stratum equipped with a level $3$ structure, see Remark \ref{R:annoying2} above. For the reasons expressed in Remark \ref{R:annoying1}, it is easier to work in $\cH$ than in the stratum. Let us denote by $\mathbb{R}\mathcal{H}$ the corresponding finite cover of $\mathbb{R}\cH(k_1,\dots,k_{\sigma})$.

Given a translation surface $(X,\omega)$, let $\Sigma\subset X$ denote the set of zeros of $\omega$. Pick any basis $\{\xi_1, \ldots, \xi_n\}$ for the relative homology group $H_1(X,\Sigma; \bZ)$. The map $\Phi:\mathbb{R}\cH\to \bC^n$ defined by
\[\Phi(X,\omega)=\left( \int_{\xi_1} \omega, \ldots, \int_{\xi_n}\omega \right)\]
defines local \emph{period coordinates} on a neighborhood $(X,\omega)\in \mathbb{R}\cH$. %We may also refer to the \emph{absolute period coordinates} of a translation surface: these are the integrals of $\omega$ over a basis of absolute homology $H_1(X,\bZ)$.

Let $H^1$ denote the flat bundle over $\mathbb{R}\cH$ whose fiber over $(X, \omega)$ is $H^1(X;\bC)$, and let $H^1_{rel}$ denote the flat bundle whose fiber over $(X,\omega)$ is $H^1(X,\Sigma; \bC)$, where $\Sigma$ is the set of singularities of $(X,\omega)$. Let $p:H^1_{rel}\to H^1$ denote the natural map from relative to absolute cohomology. Also let $H^1_\bR$ denote the flat bundle whose fiber over $(X,\omega)$ is $H^1(X;\bR)$.

From the description of the local period coordinates, it is clear that the tangent bundle of $\mathbb{R}\cH$ is naturally isomorphic to $H^1_{rel}$:  this follows from the fact that $H^1(X,\Sigma;\mathbb{C})\simeq \textrm{Hom}(H_1(X,\Sigma;\mathbb{Z}),\mathbb{C})$.

Period coordinates provide $\mathbb{R}\cH$ with a system of coordinate charts with values in $\bC^n$ and transition maps in $GL(n,\bZ)$. An \emph{affine invariant submanifold} of $\cH$ is  an immersed manifold $\cM\hookrightarrow \cH$ such that each point of $\cM$ has a neighborhood whose image is equal to the set of unit area surfaces satisfying a set of real linear equations in  period coordinates. For notational simplicity, generally we will treat affine invariant submanifold as subsets of strata, referring to the image of the immersion.
%All linear equations in this paper are assumed to be homogeneous, i.e. have constant term 0.
%The tangent bundle $T(\cM)$ to $\mathbb{R}\cM$ where $\cM$ is an affine invariant submanifold is naturally a subbundle of $H^1_{rel}$.

An affine invariant submanifold of a stratum (rather than of a finite cover) is the projection of an affine invariant submanifold via the finite cover map from $\cH$ as above to the stratum.

%%%%%%%%%%%%%%%%%%%
%%%%%%%%%%%%%%%%%%%

\subsection{Hodge and symplectic inner products.}\label{SS:hodge}

The complex bundle $H^1$ has a natural pseudo-Hermitian intersection form given by
$$(\alpha,\beta):=\frac{i}{2}\int_X\alpha\wedge\overline{\beta}$$
for $\alpha$ and $\beta$ in a fiber $H^1(X;\mathbb{C})$ of $H^1$. This form has signature $(g,g)$ because it is positive-definite on the subbundle $H^{1,0}\subset H^1$ whose fibers are $H^{1,0}(X)$ and it is negative-definite on the subbundle $H^{0,1}\subset H^1$ whose fibers are $H^{0,1}(X)$.

By Hodge's Representation Theorem, any $c\in H^1(X;\mathbb{R})$ is the real part of a unique $h(c)\in H^{1,0}$. In the literature, the \emph{imaginary part} of $h(c)$ is called the Hodge-star $\ast c$ of $c$ and the operator $c\mapsto\ast c$ is the \emph{Hodge-star operator}. This operator gives an isomorphism $H^{1,0}(X)\simeq H^1(X;\mathbb{R})$ and thus the restriction of the pseudo-Hermitian form $(\cdot,\cdot)$ to $H^{1,0}(X)$ induces an inner product  on the fibers of $H^1_{\mathbb{R}}$ called the \emph{Hodge inner product}.

%The Hodge inner product on $H^1_{\mathbb{R}}$ relates to the pseudo-Hermitian intersection form via the formula
%$$(h(c_1), h(c_2))=(c_1,c_2)_{Hodge}+i\langle c_1,c_2\rangle$$
%where $\langle\cdot,\cdot\rangle$ denotes the symplectic intersection form on the fibers $H^1(X;\mathbb{R})$ of $H^1_{\mathbb{R}}$.

{Given $V\subset H^1(X;\mathbb{R})$, we denote by $V^\perp$, resp. $V^{\dagger}$, the orthogonal complement of $V$ with respect to Hodge inner product, resp. symplectic intersection form. For later use, we note that $P^{\perp}=P^{\dagger}$ whenever $P$ is a \good plane. Indeed, this holds because, from the definition, a \good plane $P$ is Hodge-star invariant and hence the desired fact is a consequence of the following lemma in \cite{FMZsurvey}:

\begin{lem}[Lemma 3.4 in \cite{FMZsurvey}]\label{L:sameperp}For $V\subset H^1(X;\mathbb{R})$, one has $V^{\perp}=V^{\dagger}$ if and only if $V$ is Hodge-star invariant, and, in this case, the subspace $V^{\perp}=V^{\dagger}$ is Hodge-star invariant.% and $(V^{1,0})^{\perp}=(V^{\perp})^{1,0}$.
\end{lem}

%%%%%%%%%%%%%%%%%%%
%%%%%%%%%%%%%%%%%%%
% LIMITS OF H-T PLANES
%%%%%%%%%%%%%%%%%%%
%%%%%%%%%%%%%%%%%%%

\section{Limits of \good planes}\label{s.HTplanes}

In this section, we prove Theorem \ref{T:manygood}. The proof is naturally divided into two steps in the following two propositions.

\begin{prop}\label{P:goodonT}
Let $M$ be a translation surface on a Teichm\"uller curve. Then $M$ has at least $k$ orthogonal \good planes, where $k$ is the degree of the trace field.
\end{prop}

\begin{prop}\label{P:limit_is_good}
Suppose that $M_n\to M$ are translation surfaces, and that $P_n$ is a \good plane at $M_n$. Let $P$ be a limit of the $P_n$. Then $P$ is a \good plane at $M$.
\end{prop}

We emphasize that convergence $M_n\to M$ is in the moduli space of Abelian differentials, not in any specific stratum. Convergence $M_n\to M$ is equivalent to the existence of maps (diffeomorphisms) $f_n:M_n\to M$ such that $(f_n)_*(\omega_n)\to \omega$. The maps $f_n$ may not be uniquely defined, even only up to homotopy, because $M$ may have affine self maps with identity derivative. We say that $P$ is a limit of the $P_n$ if, for some choice of maps $f_n$, the planes $(f_n)_*(P_n)$ converge to $P$ in the Grassmannian of planes of $H^1(X,\bR)$.

\begin{proof}[\eb{Proof of Theorem \ref{T:manygood} using Propositions \ref{P:goodonT} and \ref{P:limit_is_good}.}]
%Suppose $\cM$ is an affine invariant submanifold which contains a dense set of Teichm\"uller curves whose trace fields have degree at least $k$. Then every translation surface in $\cM$ has $k$ orthogonal \good planes.
Let $M\in\cM$ be a translation surface. By the assumption of density, there are translation surfaces $M_n$ so that $M_n\to M$ and each $M_n$ lies on a Teichm\"uller curve whose trace field has degree at least $k$. By Proposition \ref{P:goodonT}, each $M_n$ has at least $k$ orthogonal \good planes, say $P_n^{(1)}, \ldots, P_n^{(k)}$. After passing to a subsequence, we may assume that for each $i=1, \ldots, k$, the planes $P_n^{(i)}$ converge to a plane $P^{(i)}$ at $M$. By continuity of the Hodge inner product, the $P^{(i)}, i=1, \ldots, k$ are Hodge orthogonal, and by  Proposition \ref{P:limit_is_good}, these planes are \good. Since $M\in\cM$ was arbitrary this proves the result.
\end{proof}

%%%%%%%%%%%%%%%%%%%
%%%%%%%%%%%%%%%%%%%

\subsection{M\"oller's description of the VHS over a Teichm\"uller curve.}  Recall that over any Teichm\"uller curve $\cM$ the bundle $H^1_\bR$ in fact has the structure of a real variation of Hodge structures (VHS). (For a concrete introduction to VHS in the context of Teichm\"uller curves, see for example \cite{W1}.)

\begin{thm}[M\"oller \cite{M}]\label{T:moller}
Suppose that $\cM$ is a lift of a Teichm\"uller curve to a finite cover $\cH$ of a stratum as above, and suppose the trace field has degree $k$. Then there is a decomposition of VHS over $\cM$
$$H^1_\bR=\bigoplus_{i=1}^k \bL_i \bigoplus \bW,$$
where $\bL_1$ is the tautological bundle and the $\bL_i, i=2, \ldots, k$ are Galois conjugate to $\bL_i$. The $\bL_i$'s are symplectically and Hodge orthogonal.
\end{thm}

%If necessary, add a "proof" which references [M] and fills in orthogonality.

The statement that this is a decomposition of VHS means precisely that each of the $k+1$ subbundles in this decomposition is a flat subbundle of $H^1_\bR$ which is the sum of of it's $(1,0)$ and $(0,1)$ parts. In particular, each fiber of each $\bL_i$ contains a holomorphic one form.

\begin{proof}[\eb{Proof of Proposition \ref{P:goodonT}.}]
The fiber of $\bL_i$ over a point $(X,\omega)\in \cM$ is a plane $P_i$ in $H^1(X,\bR)$. Any parallel transport of this plane to any other translation surface in $\cM$ again gives a fiber $P'$ of $\bL_i$. The complexification of this fiber always contains a holomorphic one form. So we conclude that each fiber of each $\bL_i$ is a \good plane.
\end{proof}

\begin{rem}
Using the techniques of the next section, it is possible to show that for a Teichm\"uller curve the $\bL_i$ are the {only} \good planes in $\bigoplus_{i=1}^k \bL_i$. In particular, an algebraically primitive Teichm\"uller curve has no \good planes except those given by Proposition \ref{P:goodonT}. We will not need this result, so its proof is omitted.
\end{rem}

%%%%%%%%%%%%%%%%%%%
%%%%%%%%%%%%%%%%%%%

\subsection{Limits of \good planes.}\label{SS:lims}

\begin{lem}\label{L:planelimit}
If $M_n\to M$, and $P_n$ a plane at $M_n$ so that $(P_n)_\bC\cap H^{1,0}(M_n)\neq \{0\}$, and $P$ is a limit of the $P_n$ at $M$, then $P_\bC\cap H^{1,0}(M)\neq \{0\}$.
\end{lem}

\begin{proof}
This is true because the space of Abelian differentials inside of first cohomology varies continuously with the complex structure on the surface.
\end{proof}

\begin{lem}\label{L:glimit}
If $M_n\to M$ are marked translation surfaces, and $P_n$ are planes at $M_n$ and $P$ is a limit of the $P_n$ at $M$, then for any $\newh \in SL(2,\bR)$ $\newh (P)$ is a limit of the $\newh (P_n)$.
\end{lem}

The definition of convergence of planes in the cohomology of marked translation surfaces is identical to the unmarked cases discussed near the beginning of this section, except that the maps $f_n$ should be required to commute with the markings. 

\begin{proof}
By definition there are maps $f_n:(X_n, \Sigma_n)\to (X,\Sigma)$ so that $(f_n)_*[\omega_n]\to [\omega].$ Fix $\newh $. For any (marked) translation surface $M$, $\newh $ induces an affine map $\phi_{\newh }: M\to \newh M$ (and a marking on $\newh M$). One checks that the maps $\phi_{\newh }\circ f_n \circ \phi_{\newh }^{-1}: \newh M_n\to \newh M$ induce maps on cohomology which send $\newh P_n$ to a sequence of planes at $\newh M$ converging to $\newh P$.
\end{proof}

\begin{proof}[\eb{Proof of Proposition \ref{P:limit_is_good}.}]
Fix $\newh $. It suffices to show that $(\newh  P)_\bC$ intersects $H^{1,0}(\newh M)$. This follows directly from the previous two lemmas. By Lemma \ref{L:glimit}, $\newh  P$ is a limit of the $\newh  P_n$. The  complexifications of the $\newh  P_n$ each intersect $H^{1,0}(\newh M_n)$, so Lemma \ref{L:planelimit} may be applied.
\end{proof}

%%%%%%%%%%%%%%%%%%%
%%%%%%%%%%%%%%%%%%%
% S-T SURFACES
%%%%%%%%%%%%%%%%%%%
%%%%%%%%%%%%%%%%%%%

\section{Square-tiled surfaces with no non-trivial good planes}\label{s.4}

In this section, we prove Theorem \ref{T:basecase}.

%%%%%%%%%%%%%%%%%%%
%%%%%%%%%%%%%%%%%%%

\subsection{Definitions.}\label{ss.4.def} Recall that the Veech group $SL(M)$ of $M=(X,\omega)$ is the stabilizer of $M$ in $SL(2,\mathbb{R})$. Denote by $\textrm{Aff}(M)$ the set of affine diffeomorphisms of $M$, i.e., the set of homeomorphisms of $M$ preserving the set of zeroes  and have the form of an affine transformation in the charts of the translation atlas. The Veech group $SL(M)$ is the set of linear parts (derivatives) of the elements of $\textrm{Aff}(M)$. There is  the following exact sequence:
$$1\to \textrm{Aut}(M)\to \textrm{Aff}(M)\to SL(M)\to 1.$$
The affine group $\textrm{Aff}(M)$ acts naturally (i.e., preserving the natural symplectic intersection form) on the homology (resp. cohomology) group $H_1(M,\mathbb{R})$ (resp. $H^1(M,\mathbb{R})$) of $M$, inducing natural homomorphisms
$$\textrm{Aff}(M)\to \textrm{Sp}(H_1(M,\mathbb{R})) \quad (\textrm{resp.}, \, \textrm{Aff}(M)\to \textrm{Sp}(H^1(M,\mathbb{R})).$$

Given a translation surface $M=(X,\omega)$, denote by $\Gamma(M)$  the image
of $\textrm{Aff}(M)$ under the natural map $\textrm{Aff}(M)\to \textrm{Sp}(H^1(M,\mathbb{R}))$. Let $\overline{\Gamma}(M)$ denote the Zariski closure of $\Gamma(M)$.

Let $H^1_\perp(M, \bR)$ denote $\spn_\bR(\Re(\omega), \Im(\omega))^\perp$ in $H^1(M,\mathbb{R})$. The perp may be taken with either the symplectic form or the Hodge inner product, and the result is the same (Lemma \ref{L:sameperp}).
Let $\Gp(M)$ be the restriction of $\Gamma(M)$ to $H^1_\perp(M, \bR)$, and let $\oGp(M)$ be the Zariski closure of $\Gp(M)$.

\begin{rem}\label{R:monodromy-duality} Define $H_1^{\perp}(M,\mathbb{R})\subset H_1(M,\mathbb{R})$ to be the annihilator of the tautological plane $\spn_\bR(\Re(\omega), \Im(\omega))\subset H^1(M,\mathbb{R})$. Define the (homology) tautological plane $H_1^{st}(M,\mathbb{R})\subset H_1(M,\mathbb{R})$ to be the the orthogonal complement of $H_1^\perp(M,\mathbb{R})$ \cite{MY}. The affine group $\textrm{Aff}(M)$ respects  Poincar\'e duality: The  representations of $\textrm{Aff}(M)$ in homology and cohomology are in a precise sense dual. If $A\in \textrm{Aff}(M)$ acts through a symplectic matrix $M(A)$ in a given symplectic basis of $H_1(M,\mathbb{R})$, then $A\in\textrm{Aff}(M)$ acts through the matrix $(M(A)^{-1})^t$ in the dual basis in $H^1(M,\mathbb{R})$. It follows that we can compute the groups $\Gamma(M)$, $\Gamma_{\perp}(M)$ and their Zariski closures by considering the action of $\textrm{Aff}(M)$ on $H_1(M,\mathbb{R})$ and $H_1^{\perp}(M,\mathbb{R})$ (i.e., we can use homology instead of cohomology). We will use this fact later, in the proof of Proposition \ref{P:st} below.
\end{rem}

%%%%%%%%%%%%%%%%%%%
%%%%%%%%%%%%%%%%%%%

\subsection{Outline and strategy.}

\begin{prop}\label{P:goodinv}
Let $M$ be a translation surface. The set of \good planes at $M$ is invariant under  $\overline{\Gamma}(M)$.
\end{prop}

\begin{prop}\label{P:st}
In each of $\cH(4)^{hyp}$ and $\cH(4)^{odd}$, there is a translation surface $M$ with $\oGp(M)=Sp(4,\bR)$.
\end{prop}

\begin{proof}[\eb{Proof of Theorem \ref{T:basecase} using Propositions \ref{P:goodinv} and \ref{P:st}.}]
Let $M$ be either of the translation surfaces given by Proposition \ref{P:st}. Our goal is to show that the only \good plane at $M$ is the tautological plane.

Suppose, in order to find a contradiction, that there is a \good plane $P_0$ at $M$ which is not the tautological plane. Consider the projection map $\rho: H^1(M,\bR)\to H^1_\perp(M,\bR)$. Set $P=\rho(P_0)$. Since the kernel of $\rho$ is exactly the tautological plane, we see that the subspace $P$ is at least one dimensional.
 Since the complexification of $\rho$ preserves $H^{1,0}(M)$ and $H^{0,1}(M)$
%Prove this in section 2 in subsection on hodge and symplectic inner products
we see that the complexification of $P$ intersects at least one of $H^{1,0}(M)$ or $H^{0,1}(M)$. Since $P$ is a real subspace, and $H^{1,0}(M)$ and $H^{0,1}(M)$ are complex conjugate, it follows that $P_\bC=P\otimes\mathbb{C}$ intersects both of $H^{1,0}(M)$ and $H^{0,1}(M)$. In particular, $P$ is two dimensional.

For all $\newh \in SL(2,\bR)$, there is the projection map $\rho: H^1(\newh M,\bR)\to H^1_\perp(\newh M,\bR)$ to the orthogonal complement of the tautological plane.

Define $\newh P\subset H^1(\newh M,\bR)$ as in \ref{SS:lims}. Since $P_0$ is a \good plane, $\newh P_0$ intersects both  $H^{1,0}(\newh M)$ and $H^{0,1}(\newh M)$. Since $\rho$ preserves $H^{1,0}(\newh M)$ and $H^{0,1}(\newh M)$, as above we see that $\newh P = \rho(\newh P_0)$ also intersects both $H^{1,0}(\newh M)$ and $H^{0,1}(\newh M)$. Hence $P$ is a \good plane. 

By  Proposition \ref{P:goodinv}, it follows that  $\gamma(P)$ is a \good plane for all $\gamma$ in the full symplectic group $Sp(H^1_\perp(X,\bR))$ on $H^1_\perp(X,\bR)$.

However, the full symplectic group acts transitively on the set of symplectic planes, and not all symplectic planes in $H^1_\perp(X,\bR)$ are \good. This is a contradiction.

To see that not all symplectic planes in $H^1_\perp(X,\bR)$ are \good, simply note that the complexification of most symplectic planes do not intersect $H^{1,0}$ or $H^{0,1}$. The set of symplectic planes is open in the set of planes, whereas the set of planes whose complexification intersects $H^{1,0}$ has positive codimension.
\end{proof}

%%%%%%%%%%%%%%%%%%%
%%%%%%%%%%%%%%%%%%%

\subsection{The variety of \good planes.} Here we prove Proposition \ref{P:goodinv}. The proof also shows that the set of \good planes at any translation surface is a projective variety.

\begin{proof}[\eb{Proof of Proposition \ref{P:goodinv}.}]
Fix a marking on $M$.
For each {$\newh \in SL(2,\bR)$}, define $H^{1,0}_{\newh }\subset H^1(M, \bC)$ by $H^{1,0}_{\newh } = \newh^{-1} (H^{1,0}(\newh M))$. Here we consider $H^{1,0}(\newh M)$ to be a subset of $H^1(\newh M, \bC)=H^1(\newh M, \bR)\otimes \bC$, and $\newh^{-1}$ acts on this subset via the usual action discussed in Section 2.  

A plane $P\subset H^1(M,\bR)$ is \good if and only if $P_\bC=P\otimes\mathbb{C}$ intersects $H^{1,0}_{\newh }$ non-trivially for all {$\newh \in SL(2,\bR)$}.

For each $\newh $, the condition that $P_\bC$ intersects $H^{1,0}_{\newh }$ non-trivially is a polynomial condition on the Grassmanian of planes $P$ in $H^1(X,\bR)$. (List a basis for $H^{1,0}_{\newh }$ next to a basis of $P$. This matrix should have rank at most $g+1$, which is equivalent to the determinant of every $g+2$ by $g+2$ minor being 0.)

Thus the set $\cP$ of \good planes at $M$ is defined by polynomial equations, finitely many for each $\newh $. (Thus in total, there are uncountably many polynomial equations, but this is not relevant.) In other words, $\cP$ is a subvariety of the Grassmanian of planes in $H^1(M,\bR)$. It's stabilizer contains $\Gamma(M)$, and hence the Zariski closure $\overline{\Gamma}(M)$. By definition, anything in the stabilizer of $\cP$ in $\textrm{Sp}(H^1(X,\bR))$ preserves the set $\cP$ of \good planes.
\end{proof}

%%%%%%%%%%%%%%%%%%%
%%%%%%%%%%%%%%%%%%%

\subsection{Square-tiled surfaces.} A translation surface $M=(X,\omega)$ is called a \emph{square-tiled surface} if there is a ramified finite cover $\pi:X\to\mathbb{R}^2/\mathbb{Z}^2$ unramified outside $0\in \mathbb{R}^2/\mathbb{Z}^2$ such that $\omega=\pi^*(dz)$. Equivalently, $(X,\omega)$ is a square-tiled surface if the relative periods of $\omega$ (i.e., the integrals of $\omega$ along paths joining two of its zeroes) belong to $\mathbb{Z}\oplus i\mathbb{Z}$.

Combinatorially, a square-tiled surface $(X,\omega)$ associated to a covering $\pi:X\to\mathbb{R}^2/\mathbb{Z}^2$ of degree $N$ is determined by a pair of permutations $h,v\in S_N$ acting transitively on $\{1,\dots, N\}$. Given a square-tiled surface, we may fix a numbering of the $N$ squares, i.e., a bijection from the set of squares to $\{1,\dots, N\}$. We may then define a pair of permutations $h,v\in S_N$ by declaring that $h(i)$, resp. $v(i)$ is the number of the square to the right, resp. on the top, of the $i$-th square.

Given $M=(X,\omega)$ a square-tiled surface associated to $\pi:M\to\mathbb{R}^2/\mathbb{Z}^2$, it is a simple task to identify the subspace $H_1^{\perp}(M,\mathbb{R})$. Indeed, as it is shown in \cite{MY}, $H_1^{\perp}(M,\mathbb{R})$ consists of all absolute homology classes on $M$ projecting to $0$ under the covering map $\pi:M\to\mathbb{R}^2/\mathbb{Z}^2$. We will use this elementary fact to compute the action of the affine group $\textrm{Aff}(M)$ on $H_1^\perp(M,\mathbb{R})$ for certain square-tiled surfaces of genus $3$.

%%%%%%%%%%%%%%%%%%%
%%%%%%%%%%%%%%%%%%%

\subsection{Connected components of $\cH(4)$.} Using a combinatorial object (extended Rauzy classes), Veech \cite{V90} showed that the minimal stratum $\cH(4)$ of genus $3$ translation surfaces has two connected components. More recently, Kontsevich and Zorich \cite{KZ} classified all connected components of all strata of genus $g\geq 2$ translation surfaces. For our purposes, it suffices to know that the two (\emph{hyperelliptic} and \emph{odd}) connected components $\cH(4)^{\textrm{hyp}}$ and $\cH(4)^{\textrm{odd}}$ can be distinguished via the parity of the spin structure.

Given a translation surface $M=(X,\omega)$ of genus $g\geq 1$, we select a canonical symplectic basis $\{\alpha_i, \beta_i: i=1,\dots, g\}$ and we compute the Arf invariant
\begin{eqnarray*}
&&\Phi(M, \{\alpha_i,\beta_i:i=1,\dots,g\})\\&=&\sum\limits_{i=1}^g (\textrm{ind}_{\omega}(\alpha_i)+1)(\textrm{ind}_{\omega}(\beta_i)+1) \, (\textrm{mod }2) ,
\end{eqnarray*}
where $\textrm{ind}_{\omega}(\gamma)$ is the degree of the Gauss map associated to the tangents of a curve $\gamma$ not passing through the zeroes of $\omega$. It can be shown that the Arf invariant is independent of the choice of canonical basis of homology and thus it defines a quantity $\Phi(M)\in\bZ/(2\bZ)$ that is called the parity of the spin structure of $M$.

In our setting, $M\in\cH(4)^{\textrm{hyp}}$, resp. $M\in\cH(4)^{\textrm{odd}}$, if and only if $\Phi(M)=0$, resp. $\Phi(M)=1$.

\begin{rem}Alternatively, one can show that $M\in\cH(4)^{\textrm{hyp}}$ if and only if $M$ admits an involution (i.e., an element of its affine group $\textrm{Aff}(M)$ with derivative $-Id$) having exactly $8$ fixed points.
\end{rem}

%Closing this short subsection, let us recall that a translation surface $M$ in a minimal stratum $\cH(2g-2)$ does not have non-trivial automorphisms. For instance, this is shown in \cite{MMY} by noticing that the quotient of $M$ by a non-trivial automorphism gives a cyclic Galois cover of a translation surface $N$ in some  minimal stratum of strictly smaller genus. On the other hand, since a small loop around the unique singularity of $N$ is a product of commutators in the fundamental group, one would deduce that the covering $M\to N$ is unramified, and, a fortiori, $M$ has more than one singularity, a contradiction.\ann{Confusing!}

%In particular, by recalling the exact sequence
%$$1\to \textrm{Aut}(M,\omega)\to \textrm{Aff}(M,\omega)\to \textrm{SL}(M,\omega)\to 1,$$
%we deduce that, for a translation surface $M\in\cH(2g-2)$, one has $\textrm{Aff}(M,\omega)\simeq\textrm{SL}(M,\omega)$, that is, affine diffeomorphisms of $M$ are completely determined by their linear parts. Therefore, for $(M,\omega)\in\cH(2g-2)$, it makes sense to talk about \emph{the} affine diffeomorphism with a given linear part $g\in SL(M,\omega)$.

%%%%%%%%%%%%%%%%%%%
%%%%%%%%%%%%%%%%%%%

\subsection{Monodromy of a square-tiled surface in $\cH(4)^{\textrm{odd}}$.} Let $M_{\ast}\in \cH(4)^{\textrm{odd}}$ be the square-tiled surface associated to the pair of permutations $h_{\ast}=(1)(2,3)(4,5,6)$ and $v_{\ast}=(1,4,2)(3,5)(6)$ (see Figure \ref{F:H4odd}).

\begin{figure}[h]
\def\svgwidth{1.09\columnwidth}
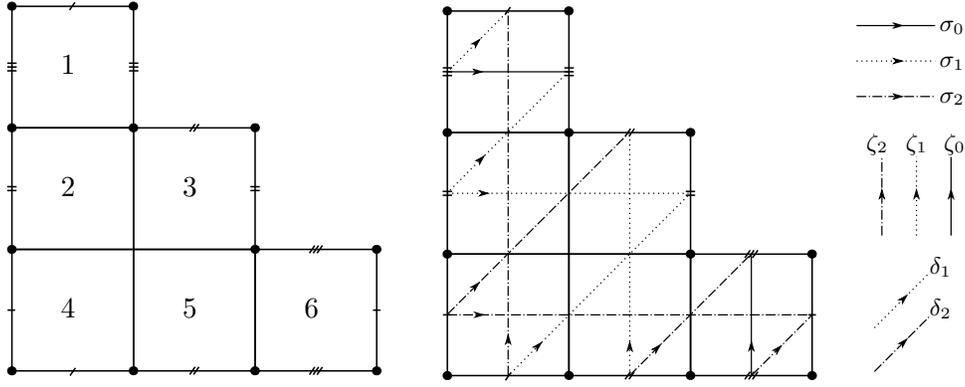
\caption{The square-tiled surface $M_\ast$ (left), redrawn (center) with the curves referred to below, and index (right).}
\label{F:H4odd}
\end{figure}

\begin{lem}\label{L:ABC}
Working in a suitable symplectic basis, $\Gp(M_\ast)\subset Sp(4,\bZ)$ contains the transposes of the inverses of the matrices
\setlength\arraycolsep{2pt}
\begin{eqnarray*}
&& \tiny{A_*=\left(\begin{array}{cccc} 1 & 0 & 3 & 3 \\ 0 & 1 & -2 & -4 \\ 0 & 0 & 1 & 0 \\ 0 & 0 & 0 & 1 \end{array}\right), \quad
B_*=\left(\begin{array}{cccc} 1 & 0 & 0 & 0 \\ 0 & 1 & 0 & 0 \\ 3 & 3 & 1 & 0 \\ -2 & -4 & 0 & 1 \end{array}\right),\quad
C_*=\left(\begin{array}{cccc} 2 & 2 & 1 & 2 \\ -1 & -1 & -1 & -2 \\ -1 & -2 & 0 & -2 \\ 1 & 2 & 1 & 3\end{array}\right).}
\end{eqnarray*}
\setlength\arraycolsep{6pt}
\end{lem}

\begin{proof}
Note that $M_{\ast}$ decomposes into three horizontal, resp. vertical, cylinders whose waist curves $\sigma_0$, $\sigma_1$ and $\sigma_2$, resp. $\zeta_0$, $\zeta_1$ and $\zeta_2$, have lengths $1$, $2$ and $3$. Also, $M_{\ast}$ decomposes into two cylinders in the slope $1$ direction whose waist curves $\delta_1$ and $\delta_2$ are defined by the property that $\delta_1$ intersects $\sigma_0$ and $\delta_2$ intersects $\zeta_2$. By direct inspection, it is not hard to check that $\delta_1=\sigma_1+\sigma_0+\zeta_2$ and $\delta_2=\sigma_2+\zeta_1+\zeta_0$.

Denote by $A, B, C\in\textrm{Aff}(M_{\ast})$ be the affine diffeomorphisms of $M_{\ast}$ with linear parts $$dA=\left(\begin{array}{cc} 1 & 6 \\ 0 & 1\end{array}\right),\quad dB=\left(\begin{array}{cc} 1 & 0 \\ 6 & 1\end{array}\right),\quad dC=\left(\begin{array}{cc} -2 & 3 \\ -3 & 4\end{array}\right).$$ Geometrically, $A, B, C$ correspond to certain Dehn multitwists in the horizontal, vertical and slope $1$ directions.

The actions of $A, B, C$ on homology $H_1(M_{\ast},\mathbb{R})$ are:
\small{\begin{align*}
&A(\sigma_i)=\sigma_i&
&\textrm{for } i=1, 2, 3,& &&\\
&A(\zeta_0)=\zeta_0+2\sigma_2,&
&A(\zeta_1)=\zeta_1+3\sigma_1+2\sigma_2,&
&A(\zeta_2)=\zeta_2+3\sigma_1+2\sigma_2+6\sigma_0,&\\
&B(\sigma_0)=\sigma_0+2\zeta_2,&
&B(\sigma_1)=\sigma_1+3\zeta_1+2\zeta_2, &
&B(\sigma_2)=\sigma_2+3\zeta_1+2\zeta_2+6\zeta_0,&\\
&B(\zeta_i)=\zeta_i &
&\textrm{for } i=1, 2, 3,& &&\\
&C(\sigma_0)=\sigma_0-\delta_1,&
&C(\sigma_1)=\sigma_1-\delta_1-\delta_2, &
&C(\sigma_2)=\sigma_2-\delta_1-2\delta_2,&\\
&C(\zeta_0)=\zeta_0+\delta_2,&
&C(\zeta_1)=\zeta_1+\delta_1+\delta_2, &
&C(\zeta_2)=\zeta_2+2\delta_1+\delta_2.&
\end{align*}}
Note that $\overline{\sigma}_1:=\sigma_1-2\sigma_0$, $\overline{\sigma}_2:=\sigma_2-3\sigma_0$, $\overline{\zeta}_1:=\zeta_1-2\zeta_0$ and $\overline{\zeta}_2:=\zeta_2-3\zeta_0$ form a basis of $H_1^{\perp}(M_{\ast},\mathbb{R})$, the symplectic orthogonal of the tautological subbundle $H_1^{st}(M_{\ast},\mathbb{R})$ spanned by $\sigma=\sigma_0+\sigma_1+\sigma_2$ and $\zeta=\zeta_0+\zeta_1+\zeta_2$. In this basis, the actions of $A, B, C$ on $H_1^{\perp}(M_{\ast},\mathbb{R})$ are given by the symplectic matrices above. By Poincar\'e duality (see Remark \ref{R:monodromy-duality} above), it follows that $(A_*^{-1})^t, (B_*^{-1})^t, (C_*^{-1})^t\in \Gamma_{\perp}(M_*)$.
\end{proof}

%%%%%%%%%%%%%%%%%%%
%%%%%%%%%%%%%%%%%%%

\subsection{Monodromy of a square-tiled surface in $\cH(4)^{\textrm{hyp}}$} Let $M_{\ast\ast}\in \cH(4)^{\textrm{hyp}}$ be the square-tiled surface associated to the pair of permutations $h_{\ast\ast}=(1)(2,3)(4,5,6)$ and $v_{\ast\ast}=(1,2)(3,4)(5)(6)$ (see Figure \ref{F:H4hyp}).

\begin{figure}[h]
\def\svgwidth{1.06\columnwidth}
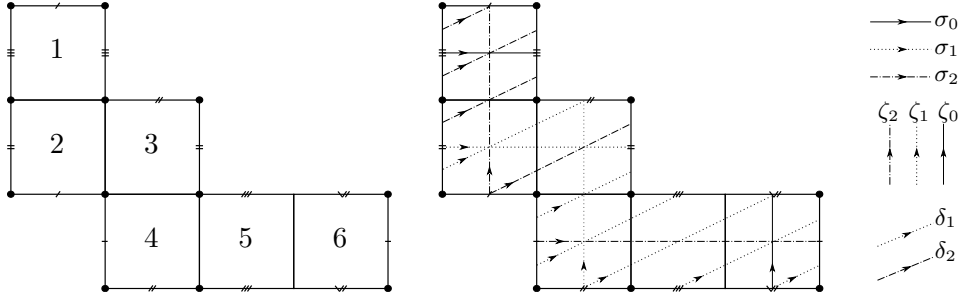
\caption{The square-tiled surface $M_{\ast\ast}$ (left), redrawn (center) with the curves referred to below, and index (right).}
\label{F:H4hyp}
\end{figure}

\begin{lem}\label{L:DEF}
Working in a suitable symplectic basis, $\Gp(M_{\ast\ast})\subset Sp(4,\bZ)$ contains the transposes of the inverses of the matrices
\setlength\arraycolsep{2pt}
\begin{eqnarray*}
&& \tiny{D_*=\left(\begin{array}{cccc} 1 & 0 & 3 & 3 \\ 0 & 1 & -2 & -4 \\ 0 & 0 & 1 & 0 \\ 0 & 0 & 0 & 1 \end{array}\right), \quad
E_*=\left(\begin{array}{cccc} 1 & 0 & 0 & 0 \\ 0 & 1 & 0 & 0 \\ 1 & 1 & 1 & 0 \\ -1 & -3 & 0 & 1 \end{array}\right),\quad
F_*=\left(\begin{array}{cccc} 2 & 3 & 1 & 3 \\ -2 & -5 & -2 & -6 \\ -1 & -3 & 0 & -3 \\ 2 & 6 & 2 & 7\end{array}\right).}
\end{eqnarray*}
\setlength\arraycolsep{6pt}
\end{lem}

\begin{proof}
Note that $M_{\ast\ast}$ decomposes into three horizontal, resp. vertical, cylinders whose waist curves $\sigma_0$, $\sigma_1$ and $\sigma_2$, resp. $\zeta_0$, $\zeta_1$ and $\zeta_2$, have lengths $1$, $2$ and $3$, resp. $1$, $2$ and $2$. Also, $M_{\ast\ast}$ decomposes into two cylinders in the slope $1/2$ direction whose waist curves $\delta_1$ and $\delta_2$ are defined by the property that $\delta_1$ intersects $\sigma_0$ and $\delta_2$ intersects $\zeta_0$. By direct inspection, it is not hard to check that $\delta_1=\sigma_1+2\sigma_0+\zeta_2$ and $\delta_2=\sigma_1+2\sigma_2+\zeta_1+2\zeta_0$.

Denote by $D, E, F\in\textrm{Aff}(M_{\ast\ast})$ be the affine diffeomorphisms of $M_{\ast\ast}$ with linear parts $$dD=\left(\begin{array}{cc} 1 & 6 \\ 0 & 1\end{array}\right),\quad dE=\left(\begin{array}{cc} 1 & 0 \\ 2 & 1\end{array}\right),\quad dF=\left(\begin{array}{cc} -7 & 16 \\ -4 & 9\end{array}\right).$$ Geometrically, $D, E, F$ correspond to certain Dehn multitwists in the horizontal, vertical and slope $1/2$ directions.

The actions of $D, E, F$ on homology $H_1(M_{\ast\ast},\mathbb{R})$ are:
\small{
\begin{align*}
&D(\sigma_i)=\sigma_i&
&\textrm{for } i=1, 2, 3,& &&\\
&D(\zeta_0)=\zeta_0+2\sigma_2,&
&D(\zeta_1)=\zeta_1+3\sigma_1+2\sigma_2, &
&D(\zeta_2)=\zeta_2+3\sigma_1+6\sigma_0,&\\
&E(\sigma_0)=\sigma_0+\zeta_2, &
&E(\sigma_1)=\sigma_1+\zeta_1+\zeta_2,&
&E(\sigma_2)=\sigma_2+\zeta_1+4\zeta_0,&\\
&E(\zeta_i)=\zeta_i &
&\textrm{for } i=1, 2, 3,& &&\\
&F(\sigma_0)=\sigma_0-2\delta_1,&
&F(\sigma_1)=\sigma_1-2\delta_1-\delta_2, &
&F(\sigma_2)=\sigma_2-3\delta_2,&\\
&F(\zeta_0)=\zeta_0+2\delta_2, &
&F(\zeta_1)=\zeta_1+2\delta_1+3\delta_2, &
&F(\zeta_2)=\zeta_2+6\delta_1+\delta_2&
\end{align*}}
Note that $\overline{\sigma}_1:=\sigma_1-2\sigma_0$, $\overline{\sigma}_2:=\sigma_2-3\sigma_0$, $\overline{\zeta}_1:=\zeta_1-2\zeta_0$ and $\overline{\zeta}_2:=\zeta_2-2\zeta_0$ form a basis of $H_1^{\perp}(M_{\ast\ast},\mathbb{R})$, the symplectic orthogonal of the tautological subbundle $H_1^{st}(M_{\ast\ast},\mathbb{R})$ spanned by
$\sigma=\sigma_0+\sigma_1+\sigma_2$ and $\zeta=\zeta_0+\zeta_1+\zeta_2$. In this basis, the actions of $D, E, F$ on $H_1^{\perp}(M_{\ast\ast},\mathbb{R})$ are given by the symplectic matrices above. By Poincar\'e duality (see Remark \ref{R:monodromy-duality} above), it follows that $(D_*^{-1})^t, (E_*^{-1})^t, (F_*^{-1})^t\in \Gamma_{\perp}(M_{\ast\ast})$.
\end{proof}

%%%%%%%%%%%%%%%%%%%
%%%%%%%%%%%%%%%%%%%

\subsection{Zariski closure.} Here we complete the proof of Proposition \ref{P:st}, by computing the Zariski closures of $\Gp(M_\ast)$ and $\Gp(M_{\ast\ast})$.

\begin{lem}\label{L:z}
The Zariski closures of the groups $\langle A_{\ast}, B_{\ast}, C_{\ast}\rangle$ and $\langle D_{\ast}, E_{\ast}, F_{\ast}\rangle$ are both $Sp(4, \bR)$, where the matrices $A_{\ast},B_{\ast},C_{\ast},D_{\ast},E_{\ast},F_{\ast}$ are given in Lemmas \ref{L:ABC} and \ref{L:DEF}.
\end{lem}

\begin{proof}
Denote by $G=\overline{\langle A_{\ast}, B_{\ast}, C_{\ast}\rangle}^{Zariski}\subset Sp(4,\mathbb{R})$ the Zariski closure of the subgroup generated by $A_{\ast},B_{\ast}, C_{\ast}$. We claim that
$G=Sp(4,\mathbb{R})$. Indeed, denote by $\mathfrak{g}$ the Lie algebra of $G$ and note that
$$\log A_{\ast}=\left(\begin{array}{cccc} 0 & 0 & 3 & 3 \\ 0 & 0 & -2 & -4 \\ 0 & 0 & 0 & 0 \\ 0 & 0 & 0 & 0 \end{array}\right)\in\mathfrak{g}.$$
Consider the conjugates of $\log A_{\ast}$ by $$B_{\ast}, B_{\ast}^2, A_{\ast}B_{\ast}, A_{\ast}^2B_{\ast}, B_{\ast}A_{\ast}B_{\ast}, C_{\ast}, C_{\ast}^2, A_{\ast}C_{\ast}, B_{\ast}C_{\ast}.$$ These nine matrices are, respectively,
\setlength\arraycolsep{2pt}
\begin{eqnarray*}
&& \tiny{\left(\begin{array}{cccc} -3 & 3 & 3 & 3 \\ -2 & -10 & -2 & -4 \\ -15 & -21 & 3 & -3 \\ -14 & 34 & 2 & 10 \end{array}\right), \,
 \left(\begin{array}{cccc} -6 & 6 & 3 & 3 \\ -4 & -20 & -2 & -4 \\ -60 & -84 & 6 & -6 \\ 56 & 136 & 4 & 20 \end{array}\right),\,
\left(\begin{array}{cccc} -6 & 42 & 120 & 210 \\ -28 & -104 & -140 & -370 \\ -15 & -21 & 6 & -42 \\ 14 & 34 & 28 & 104 \end{array}\right),}\\
&& \tiny{\left(\begin{array}{cccc} -9 & 81 & 411 & 747 \\ -54 & -198 & -498 & -1332 \\ -15 & -21 & 9 & -81 \\ 14 & 34 & 54 & 198 \end{array}\right),\,
 \left(\begin{array}{cccc} 54 & 522 & 120 & 210 \\ -348 & -1164 & -140 & -370 \\ -999 & -2133 & -54 & -522 \\ 1422 & 3978 & 348 & 1164 \end{array}\right),\,
\left(\begin{array}{cccc} 4 & 8 & 6 & 6 \\ -2 & -4 & -3 & -3 \\ -4 & -8 & -3 & -3 \\ 4 & 8 & 3 & 3 \end{array}\right),}\\
&&\tiny{\left(\begin{array}{cccc} 16 & 32 & 17 & 25 \\ -12 & -24 & -12 & -18 \\ -16 & -32 & -14 & -22 \\ 16 & 32 & 14 & 22 \end{array}\right), \,
\left(\begin{array}{cccc} 4 & 8 & 10 & 26 \\ -10 & -20 & -19 & -59 \\ -4 & -8 & -7 & -23 \\ 4 & 8 & 7 & 23 \end{array}\right),\,
\left(\begin{array}{cccc} -2 & 14 & 6 & 6 \\ 1 & -7 & -3 & -3 \\ -4 & 10 & 6 & 6 \\ 1 & 11 & 3 & 3 \end{array}\right).}
\end{eqnarray*}
\setlength\arraycolsep{6pt}
All of these matrices must be in $\mathfrak{g}$. Furthermore, direct computation shows $\log A_{\ast}$ and these nine matrices are linearly independent. Since $Sp(4, \bR)$ has dimension 10, we conclude that $G=Sp(4,\bR)$.

Next denote by $H=\overline{\langle D_{\ast}, E_{\ast}, F_{\ast}\rangle}^{Zariski}\subset Sp(4,\mathbb{R})$ the Zariski closure of the subgroup generated by $D_{\ast},E_{\ast}, F_{\ast}$. We claim that
$H=Sp(4,\mathbb{R})$. Indeed, denote by $\mathfrak{h}$ the Lie algebra of $H$ and note that
$$\log D_{\ast}=\left(\begin{array}{cccc} 0 & 0 & 3 & 3 \\ 0 & 0 & -2 & -4 \\ 0 & 0 & 0 & 0 \\ 0 & 0 & 0 & 0 \end{array}\right)\in\mathfrak{h}.$$
Consider the conjugates of $\log D_{\ast}$ by $$E_{\ast}, E_{\ast}^2, D_{\ast}E_{\ast}, D_{\ast}E_{\ast}^2, E_{\ast}D_{\ast}E_{\ast}, F_{\ast},  F_{\ast}^2, D_{\ast}F_{\ast}, E_{\ast}F_{\ast}.$$ These nine matrices are, respectively,
\setlength\arraycolsep{2pt}
\begin{eqnarray*}
&& \tiny{\left(\begin{array}{cccc} 0 & 6 & 3 & 3 \\ -2 & -10 & -2 & -4 \\ -2 & -4 & 1 & -1 \\ 6 & 24 & 3 & 9 \end{array}\right), \,
\left(\begin{array}{cccc} 0 & 12 & 3 & 3 \\ -4 & -20 & -2 & -4 \\ -8 & -16 & 2 & -2 \\ 24 & 96 & 6 & 18 \end{array}\right),\,
\left(\begin{array}{cccc} 12 & 66 & 111 & 255 \\ -22 & -98 & -146 & -364 \\ -2 & -4 & -1 & -11 \\ 6 & 24 & 33 & 87 \end{array}\right),}\\
&& \tiny{\left(\begin{array}{cccc} 48 & 252 & 387 & 915 \\ -84 & -372 & -522 & -1308 \\ -8 & -16 & -6 & -42 \\ 24 & 96 & 126 & 330 \end{array}\right),\,
\left(\begin{array}{cccc} 156 & 720 & 111 & 255 \\ -240 & -1044 & -146 & -364 \\ -96 & -360 & -36 & -120 \\ 624 & 2664 & 360 & 924 \end{array}\right),\,
\left(\begin{array}{cccc} 12 & 36 & 12 & 30 \\ -24 & -72 & -20 & -58 \\ -15 & -45 & -12 & -36 \\ 30 & 90 & 24 & 72 \end{array}\right),}\\
&&\tiny{\left(\begin{array}{cccc} 54 & 162 & 51 & 147 \\ -108 & -324 & -98 & -292 \\ -60 & -180 & -54 & -162 \\ 120 & 360 & 108 & 324 \end{array}\right), \,
\left(\begin{array}{cccc} 57 & 171 & 219 & 651 \\ -114 & -342 & -434 & -1300 \\ -15 & -45 & -57 & -171 \\ 30 & 90 & 114 & 342 \end{array}\right),\,
\left(\begin{array}{cccc} 30 & 114 & 12 & 30 \\ -62 & -226 & -20 & -58 \\ -71 & -253 & -20 & -64 \\ 234 & 846 & 72 & 216 \end{array}\right).}
\end{eqnarray*}
\setlength\arraycolsep{6pt}
All of these matrices must be in $\mathfrak{h}$. Furthermore, direct computation shows $\log D_{\ast}$ and these nine matrices are linearly independent. Since $Sp(4, \bR)$ has dimension 10, we conclude that $H=Sp(4,\bR)$.
\end{proof}

\begin{proof}[\eb{Proof of Proposition \ref{P:st}.}]
We consider the square-tiled surfaces $M_{\ast}\in \cH(4)^{\textrm{odd}}$ and $M_{\ast\ast}\in \cH(4)^{\textrm{hyp}}$.

Lemmas \ref{L:ABC} and \ref{L:DEF} show that $\Gp(M_\ast)\supset\langle (A_{\ast}^{-1})^t,(B_{\ast}^{-1})^t,(C_{\ast}^{-1})^t\rangle$ and $\Gp(M_{\ast\ast})\supset\langle (D_{\ast}^{-1})^t,(E_{\ast}^{-1})^t,(F_{\ast}^{-1})^t\rangle$, and Lemma \ref{L:z} proves that both $\langle A_{\ast},B_{\ast},C_{\ast}\rangle$ and $\langle D_{\ast},E_{\ast},F_{\ast}\rangle$ have Zariski closure equal to $Sp(4,\bR)$. Since both $\oGp(M_\ast)$ and $\oGp(M_{\ast\ast})$ are subgroups of $Sp(4,\bR)$, it follows that they are both equal to $Sp(4,\bR)$.
\end{proof}

\begin{rem}
The computations in this section were first performed by the first author by hand,  using Mathematica only for matrix arithmetic and to verify linear independence of sets of matrices. The computations were then redone independently by the second author  by computer, using a short Maple program which computes the action of Dehn multitwists on homology.
\end{rem}

%%%%%%%%%%%%%%%%%%%
%%%%%%%%%%%%%%%%%%%
% INDUCTION STEP
%%%%%%%%%%%%%%%%%%%
%%%%%%%%%%%%%%%%%%%

\section{Inductive step}\label{s.induct}

In this section, we prove Theorem \ref{T:induct}. We begin by proving the first statement, since the proof is technically easier in this case.

\begin{prop}\label{P:induct1}
Let $\cH$ be a connected component of a stratum $\cH(k_1, \ldots, k_s)$ of genus $g$ translation surfaces, where $s>1$. Suppose that every translation surface in $\cH$ has $k$ \good planes. Then there is a connected component $\cH'$ of $\cH(2g-2)$ in which every translation surface has $k$ orthogonal \good planes.
\end{prop}

\begin{prop}[Kontsevich-Zorich]\label{P:min}
Let $\cH$ be a connected component of a stratum  of genus $g$ translation surfaces with more than one zero. Then there is a connected component $\cH'$ of $\cH(2g-2)$ which is contained in the boundary of $\cH$.
\end{prop}

Proposition \ref{P:min} is proven in \cite{KZ} (in particular, this statement is included in the proof of \cite[Prop. 4]{KZ}).

\begin{proof}[\eb{Proof of Proposition \ref{P:induct1}.}]
Let $\cH'$ be given by Proposition \ref{P:min}, and let $M\in \cH'$ be arbitrary. By Proposition \ref{P:min}, we see there is a sequence $M_n\in \cH$ with $M_n\to M$. We have assumed that each $M_n$ has $k$ orthogonal \good planes. By Lemma \ref{P:limit_is_good}, we see that $M$ has $k$ \good planes also.
\end{proof}

We now proceed to the second and final statement of Theorem \ref{T:induct}, whose proof is similar but more involved.

\begin{prop}\label{P:induct2}
Let $\cH$ be a connected component of the stratum $\cH(2g-2)$. Suppose that every translation surface in $\cH$ has $k$ \good planes. Then there is a connected component $\cH'$ of $\cH(2g-4)$ in which every translation surface has at least $k-1$ \good planes.
\end{prop}

Propositions \ref{P:induct1} and \ref{P:induct2} together prove Theorem \ref{T:induct}. To prove Proposition \ref{P:induct2}, we will need to study certain degenerations of translation surfaces to lower genus.

\subsection{Bubbling a handle.} In their classification of connected components of strata, Kontsevich-Zorich introduced a local surgery of Abelian differentials that increases genus by 1. They call their surgery, which involves splitting a zero and then gluing in a torus, ``bubbling a handle."

\bold{Splitting a zero.} Eskin-Masur-Zorich \cite{EMZ} introduced a local surgery of a translation surface known as splitting (or opening up) a zero. A cut and paste operation is performed in a sufficiently small neighborhood of a zero of order $m$ on a translation surface, producing a new translation surface which is flat isometric  to the old one outside of this neighborhood, but inside the neighborhood has a pair of zeros of orders $m'$ and $m''$ with sum $m'+m''=m$, joined by a short saddle connection with holonomy $v\in \bR^2$. We will say that the direction of $v$ in $\bR^2$ is the direction of the splitting. (The degenerate case $m''=0$ is usually allowed but we will not need it.) For more details, we highly recommend consulting \cite[Section 4.2]{KZ} or \cite[Section 8.1]{EMZ}.

\begin{figure}[h]
\includegraphics[scale=0.3]{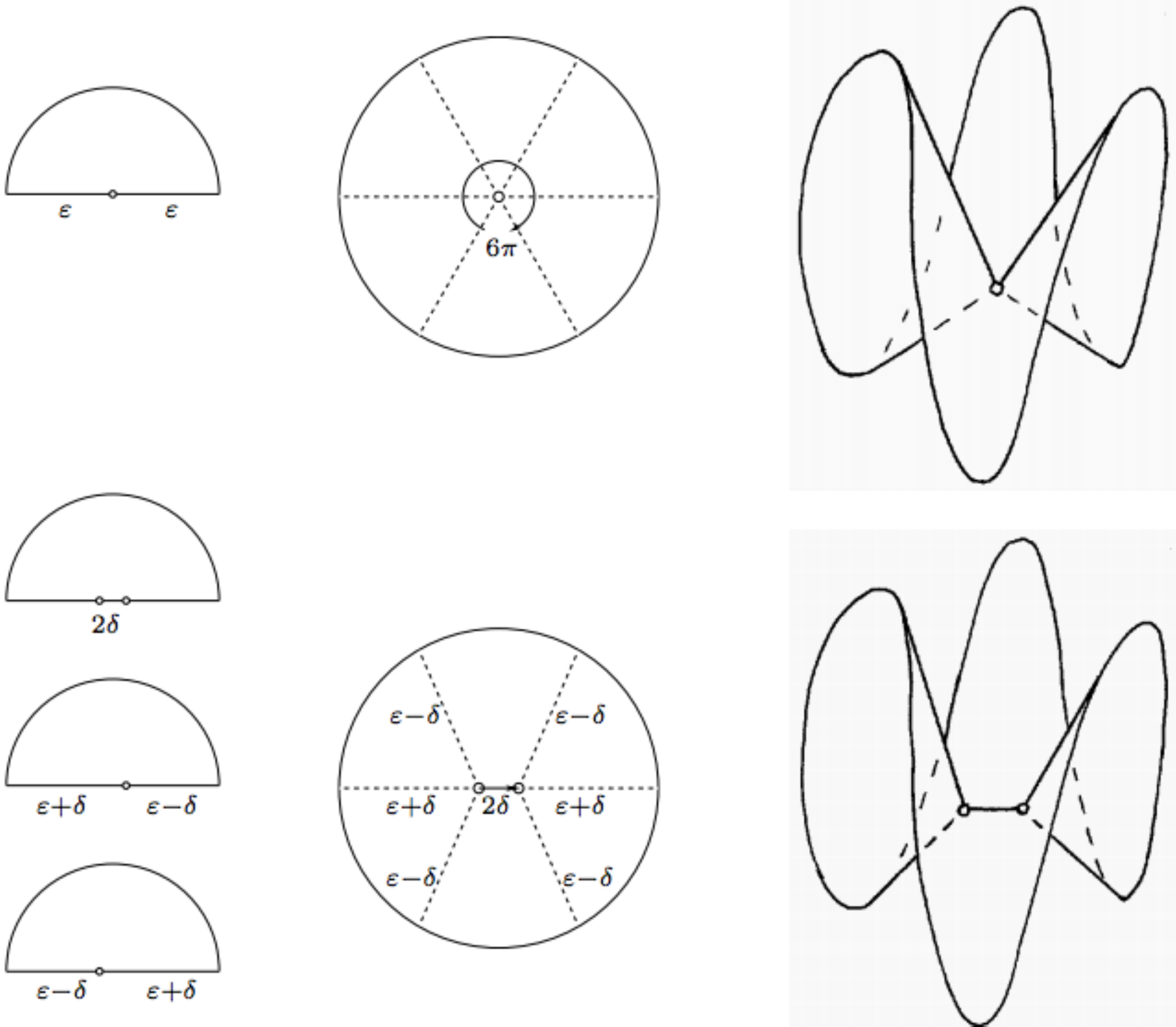}
\caption{Splitting a zero is a local surgery which replaces a single zero with a pair of zeros joined by a short saddle connection. In this picture $m=2, m'=m''=1$, and the holonomy $v$ of the newly created short saddle connection is horizontal and has length $2\delta$. Figure from \cite{EMZ, KZ}.}
\label{F:split}
\end{figure}

\bold{Gluing in the torus.} The process of bubbling a handle begins by splitting a zero. The new short saddle connection is then cut, producing a slit. A cylinder is glued into this slit, one end to each half of the slit. This identifies the two endpoints of the slit.
\begin{figure}[h]
\includegraphics[scale=0.3]{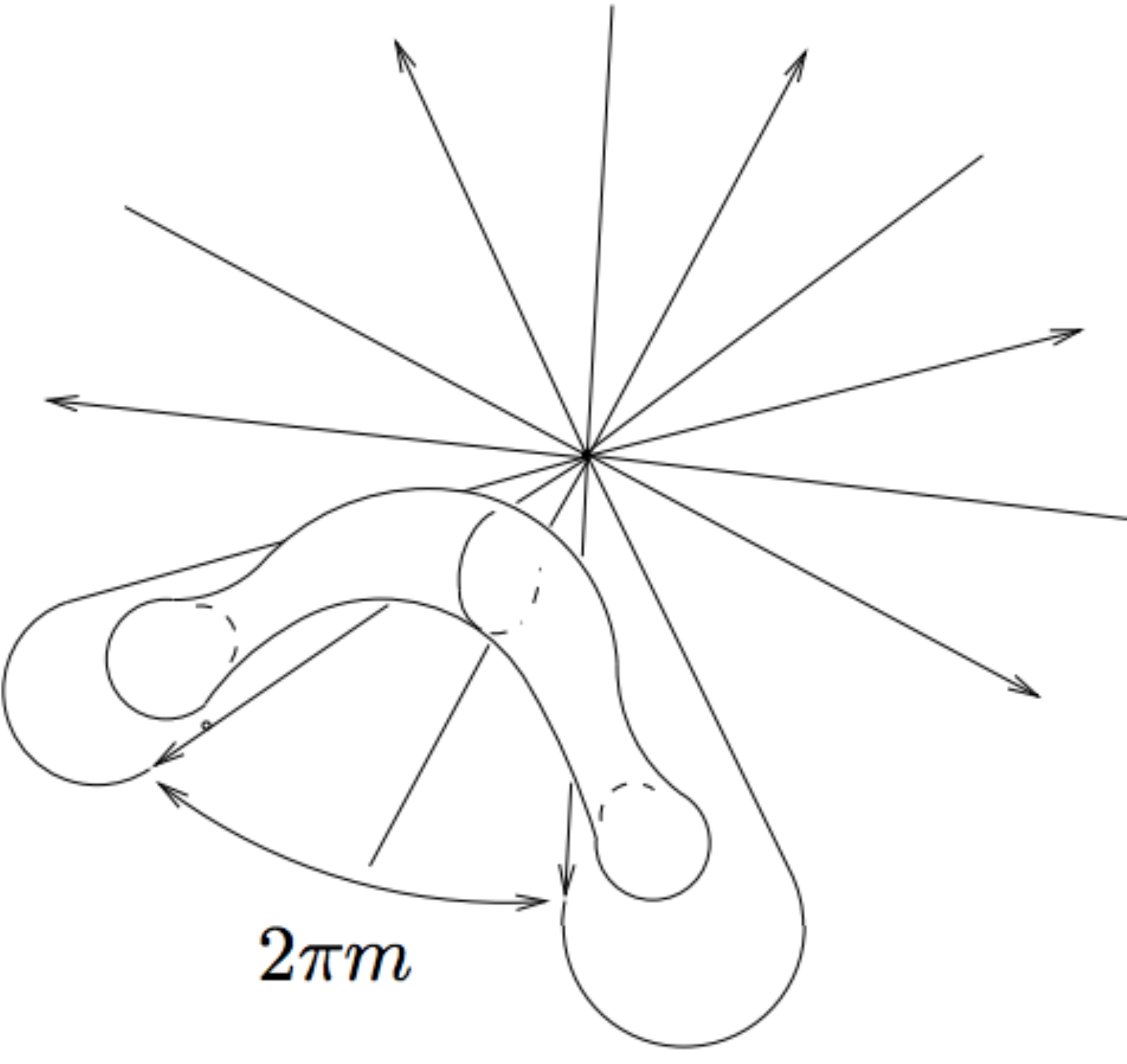}
\caption{After opening up a zero, a cylinder or handle is glued in. The process of splitting the zero and then gluing in a cylinder is called bubbling a handle. Figure from \cite{KZ}.}
\label{F:handle}
\end{figure}
The torus is called a handle, which has been ``bubbled off" from the original translation surface. The result is a translation surface with the same number of zeros, but the order of one of the zeros (the one which was split) goes up by one, and the genus increases by one. Even after the zero has been split open, there are a number of ways of gluing in a handle, parameterized by the width of the cylinder, and the twist applied when gluing. We will usually glue in a \emph{square handle}, that is, we will take a square with a side parallel to the slit, glue the two sides parallel to the slit to the two sides of the slit, and glue the two remaining sides to each other.  See figure \ref{F:square}.

\begin{figure}[h]
\includegraphics[scale=0.3]{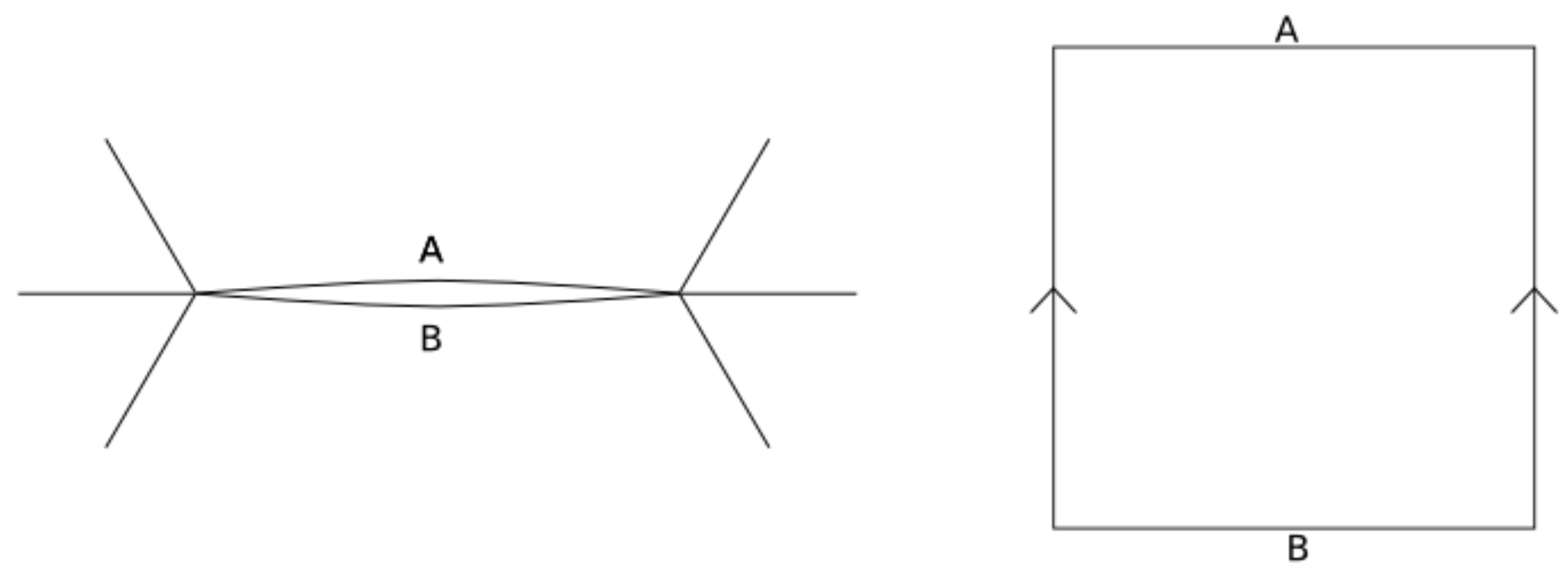}
\caption{Another picture of bubbling a square torus (in a horizontal direction). After the two zeros have been separated, the short saddle connection joining them is cut and a square cylinder is glued in. As a result of gluing in the cylinder the two endpoints of the slit get identified. (The order of zero in this picture is less than in the previous picture.)}
\label{F:square}
\end{figure}

\bold{Adjacency of strata.} The following is a key step in the Kontsevich-Zorich classification of strata.

\begin{prop}\label{P:minadj}
For each connected component $\cH$ of a minimal stratum, there is a connected component $\cH'$ of the minimal stratum in genus one less, so that from every translation surface in $\cH'$ it is possible to bubble a square handle so that the resulting translation surface is in $\cH$.
\end{prop}

\begin{proof}
Lemma 14 in \cite{KZ} asserts that given such an $\cH$, there is a connected component $\cH'$ of the minimal stratum in genus one less, and a translation surface in $\cH'$ from which it is possible to bubble a handle and obtain a translation surface in $\cH$. The continuity of the bubbling a handle operation (Lemma 10 in \cite{KZ}) then gives that the translation surface in $\cH'$ can be changed arbitrarily along paths, and the cylinder which is bubbled off can be made square.
\end{proof}

\subsection{The limit of a \good plane as the handle shrinks.} Before we discuss convergence of \good planes, we need to discuss convergence of  translation surfaces.

\begin{prop}\label{P:DM_limit}
Let $(X_1, \omega_1)$ be the result of bubbling a square handle from $(Y,\omega)$, and let $(X_n, \omega_n)$ be a sequence of translation surfaces obtained from the same bubbling construction but letting the size of the square handle go to zero. Then $X_n$ converges in the Deligne-Mumford compactification to a stable curve $X_{\infty}$ given by $Y$ joined at a node to a torus, and $\omega_n$ converges to the Abelian differential which is equal to $\omega$ on $Y$ and is zero on the torus. The node on $Y$ is at the zero of $\omega$.

Fix $h\in SL(2,\bR)$, and set $(X_n^h, \omega_n^h)=h(X_n,\omega_n)$, and let $(Y^h, \omega^h)=h(Y,\omega)$. Then $X_n^h$ converges in the Deligne-Mumford compactification to  $Y^h$ joined at a node to a torus, and $\omega_n^h$ converges to the Abelian differential which is equal to $\omega^h$ on $Y^h$ and is zero on the torus.
\end{prop}
\begin{figure}[h]
\includegraphics[scale=0.5]{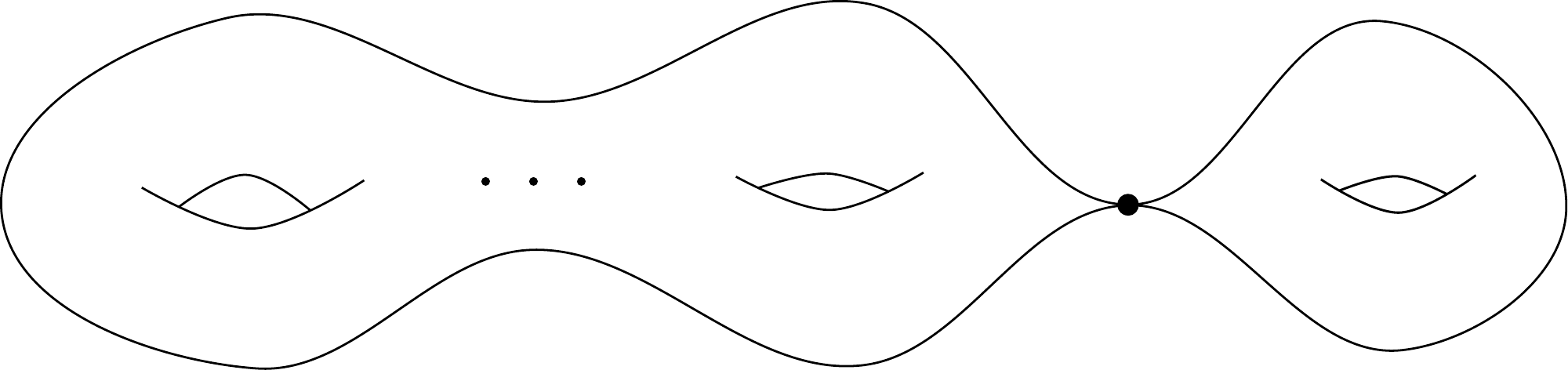}
\caption{The limit $X_\infty$ of $X_n$ is a noded Riemann surface, equal to the genus $g-1$ surface $Y$ (left) glued to a torus (right) at a node. The limit of the $\omega_n$ is supported on $Y$, where it is equal to $\omega.$}
\label{F:noded}
\end{figure}

One of the key technical points is that the torus is  smooth, i.e., no additional curve on the torus gets pinched. The proof is deferred to section \ref{S:conv}. We now proceed to limits of \good planes, proving a version of Proposition \ref{P:limit_is_good}.

\begin{prop}\label{P:limit_is_good2}
We use the notation and set up of the previous proposition. Let $\iota:Y\to X_\infty$ denote the natural inclusion. Suppose $P_n$ is a \good plane at $(X_n,\omega_n)$, and $P_n$ converges to a plane $P\subset H^1(X_\infty,\bR)$. Then $\iota^*(P)$ is a \good plane at $(Y,\omega)$, unless $\iota^*(P)$ is the zero subspace.
\end{prop}

A point of the {augmented Teichm\"uller space, resp.} Deligne-Mumford compactification is a {marked noded Riemann surface, resp.} noded Riemann surface. Topologically, a noded Riemann surface can be thought of as the result of pinching a set of disjoint curves on a surface of genus $g$. Roughly speaking, a marking on noded Riemann surface a map from the reference surface $S$ which collapses some disjoint curves. See \cite{Ba} Section 5.2 for the precise definition.  {The Deligne-Mumford compactification is obtained from augmented Teichm\"uller space by quotient by the mapping class group (or, equivalently, by forgetting markings)}. If, as is the case for $X_\infty$, only separating curves are pinched, the noded Riemann surface is said to have \emph{geometric genus} $g$.

Pinching a set of homologically trivial curves does not effect (co)homology, and so the bundle $H^1$ over {Teichm\"uller} space (whose fiber over a Riemann surface $X$ is $H^1(X,\bR)$) extends continuously to points in the {augmented Teichm\"uller space} where only separating curves have been pinched.

The bundle $H^{1,0}$ also extends continuously to all {(marked)} noded Riemann surfaces, but for us it will be easier to think only about the Riemann surfaces of geometric genus $g$. The fiber of $H^{1,0}$ over a noded Riemann surface $X$ of geometric genus $g$ is isomorphic to the space of unit area holomorphic one forms on the disconnected smooth Riemann surface obtained by separating all the nodes of $X$. In particular, $H^{1,0}(X_\infty)$ is isomorphic to the direct sum of $H^{1,0}(Y)$ and $H^{1,0}$ of the torus which is glued to $Y$.

The bundle $H^{1,0}$ is a continuous subbundle of $H^1$ over the locus of possibly noded {(marked)} Riemann surfaces of geometric genus $g$.

We will show that there are shrinking neighborhoods $U_n$ of the node on $X_\infty$, and conformal maps $f_n:X_\infty\setminus U_n\to X_n$ such that $(f_n)^*(\omega_n)\to\omega$. (This is the definition of convergence of Abelian differentials, as we will recall in the next section.) Again we say that planes $P_n\subset H^1(X_n,\bR)$ converge to a plane $P\subset H^1(X_\infty, \bR)$ if the maps $f_n$ can be chosen so that $(f_n)^*(P_n)\to P$. Again, if the surfaces are marked, these maps will be required to commute with the markings. 

We need the following version of Lemma \ref{L:planelimit}.

\begin{lem}\label{L:planelimit2}
Suppose $P_n\subset H^1(X_n,\bR)$ is a sequence of planes converging to $P\subset H^1(X_\infty, \bR)$. If $(P_n)_\bC\cap H^{1,0}(X_n)\neq \{0\}$ for all $n$, then $P_\bC\cap H^{1,0}(X_{\infty})\neq \{0\}$.
\end{lem}

\begin{proof}
This follows as in Lemma \ref{L:planelimit}, because $H^{1,0}$ is a continuous subbundle of $H^1\otimes \bC$ over the locus of possibly noded {(marked)} Riemann surfaces of geometric genus $g$.
\end{proof}

We also need the following version of Lemma \ref{L:glimit}, whose proof is also deferred to the end of the next section.

\begin{lem}\label{L:glimit2}
Suppose $P_n\subset H^1(X_n,\bR)$ is a sequence of planes converging to $P\subset H^1(X_\infty, \bR)$. Fix any {$\newh \in SL(2,\bR)$}. Then $\newh  (P_n)$ converges to $P'$, where $P'$ satisfies $\newh (\iota^*(P))=\iota^*(P')$.
\end{lem}

%\begin{prop}\label{P:induct2}
%Let $\cH$ be a connected component of the stratum $\cH(2g-2)$. Suppose that every translation surface in $\cH$ has $g$ \good planes. Then there is a connected component $\cH'$ of $\cH(2g-4)$ in which every translation surface has at least $g-1$ \good planes.
%\end{prop}
%
%\begin{prop}\label{P:minadj}
%For each connected component $\cH$ of a minimal stratum, there is a connected component $\cH'$ of the minimal stratum in genus one less, so that from every translation surface in $\cH'$ it is possible to bubble a square handle so that the resulting translation surface is in $\cH$.
%\end{prop}

\begin{proof}[\eb{Proof of Proposition \ref{P:limit_is_good2}.}]
This follows directly from Lemmas \ref{L:planelimit2} and \ref{L:glimit2}.
\end{proof}

\begin{proof}[\eb{Proof of Proposition \ref{P:induct2}.}]
Let $\cH$ be a connected component of the stratum $\cH(2g-2)$, and let $\cH'$ be the connected component of $\cH(2g-4)$ given by Proposition \ref{P:minadj}. We are assuming that every translation surface in $\cH$ has $k$ \good planes.

Let $(Y,\omega)\in \cH'$. Let $(X_1, \omega_1)$ be the result of bubbling a square handle from $(Y,\omega)$, and let $(X_n, \omega_n)$ be a sequence of translation surfaces obtained from the same bubbling construction but letting the size of the square handle go to zero. By Proposition \ref{P:DM_limit}, given $h\in SL(2,\bR)$, one has that $X_n^{h}$ converges in the Deligne-Mumford compactification to  $Y^{h}$ joined at a node to a torus, and $\omega_n^{h}$ converges to the Abelian differential which is equal to $\omega^{h}$ on $Y^{h}$ and is zero on the torus.

Let $P_n^{(1)}, \ldots, P_n^{(k)}$ be orthogonal \good planes at $(X_n,\omega_n)$. After passing to a subsequence, we may assume that $P_n^{(i)}$ converges to a plane $P^{(i)}$ at $X_\infty$.

The Hodge inner product extends continuously to nodal surfaces of genus $g$. (The definition given in section \ref{SS:hodge} is equally valid for nodal surfaces of genus $g$. One could also use the symplectic pairing here.) The planes $P^{(i)}$ are Hodge orthogonal at $X_\infty$, and hence at most one can be contained in $\ker(\iota^*)$. Without loss of generality, assume that $P^{(i)}$ is not contained in $\ker(\iota^*)$ for $i<k$.

Then Proposition \ref{P:limit_is_good2} gives that the set of $\iota^*(P^{(i)})$ for $1\leq i\leq k-1$ are orthogonal \good planes at $(Y,\omega)$.
\end{proof}

\begin{rem}
In this proof we see the theoretical possibility of loosing a \good plane to $\ker(\iota^*)$ in the inductive step. It is precisely this which has stopped us from resolving Problem \ref{p1} below.
\end{rem}

\section{Proof of Proposition \ref{P:DM_limit}}\label{S:conv}

We will proceed in two steps. First we will build the candidate $X_\infty$ using flat geometry, and then we will prove that $X_n$ converges to this $X_\infty$. Once the proof of Proposition \ref{P:DM_limit} is complete, we will sketch the proof of Lemma \ref{L:glimit2}.

The techniques and results in this section are not new: compare to \cite[Section 4]{EKZbig}, which uses work of Rafi \cite{R}.

\subsection{Construction of $X_\infty$.} The proposition states that $X_\infty$ is $Y$ joined at a node to a torus. We will specify the torus using flat geometry. More precisely, we will construct an infinite volume flat structure which is given by a meromorphic Abelian differential on a torus, and it is this torus that we will join to $Y$ at a node to construct $X_\infty$. The flat infinite volume torus will be obtained by bubbling a handle from an infinite volume flat surface on the sphere.

We begin by constructing the infinite volume flat surface on the sphere, with a zero of some order $m>0$. Take $m+1$ copies of the upper half plane, and $m+1$ copies of the lower half plane. Glue them to each other along the segments $(-\infty,0]$ and $[0,\infty)$, to construct an infinite volume flat surface with one singularity of cone angle $(m+1)2\pi$. This is in fact a meromorphic Abelian differential on the sphere; there is a pole at infinity.

\begin{figure}[h]
\includegraphics[scale=0.29]{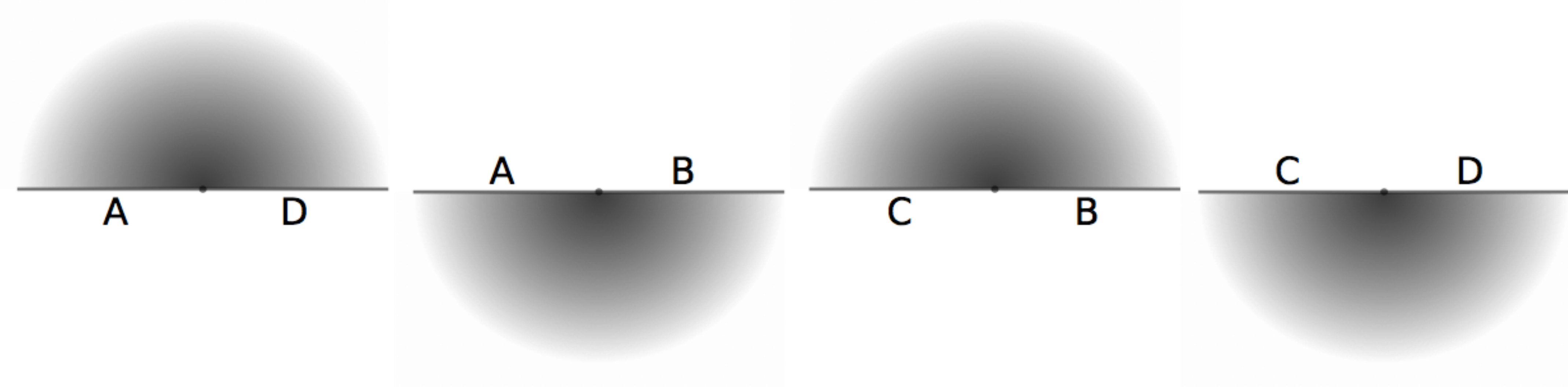}
\caption{When $m=1$, we take two copies each of the upper and lower half planes, and glue them together along the lines $(-\infty,0]$ and $[0,\infty)$ to create a single cone point with angle $4\pi$. (The shading indicates either the upper or lower half plane, and each labelled line segment is either $(-\infty,0]$ or $[0,\infty)$.) }
\label{F:zero1}
\end{figure}

From this flat structure we bubble a square handle of size 1 to obtain an infinite volume flat structure on a torus $T$. The handle should be bubbled in the same way as on $Y$: when splitting the zero there is a finite choice of how to split the zero into two zeros of order $m'$ and $m''$ with $m'+m''=m$. This choice should be made the same way for the bubbling of the infinite volume sphere as for $Y$.

\begin{figure}[h]
\includegraphics[scale=0.45]{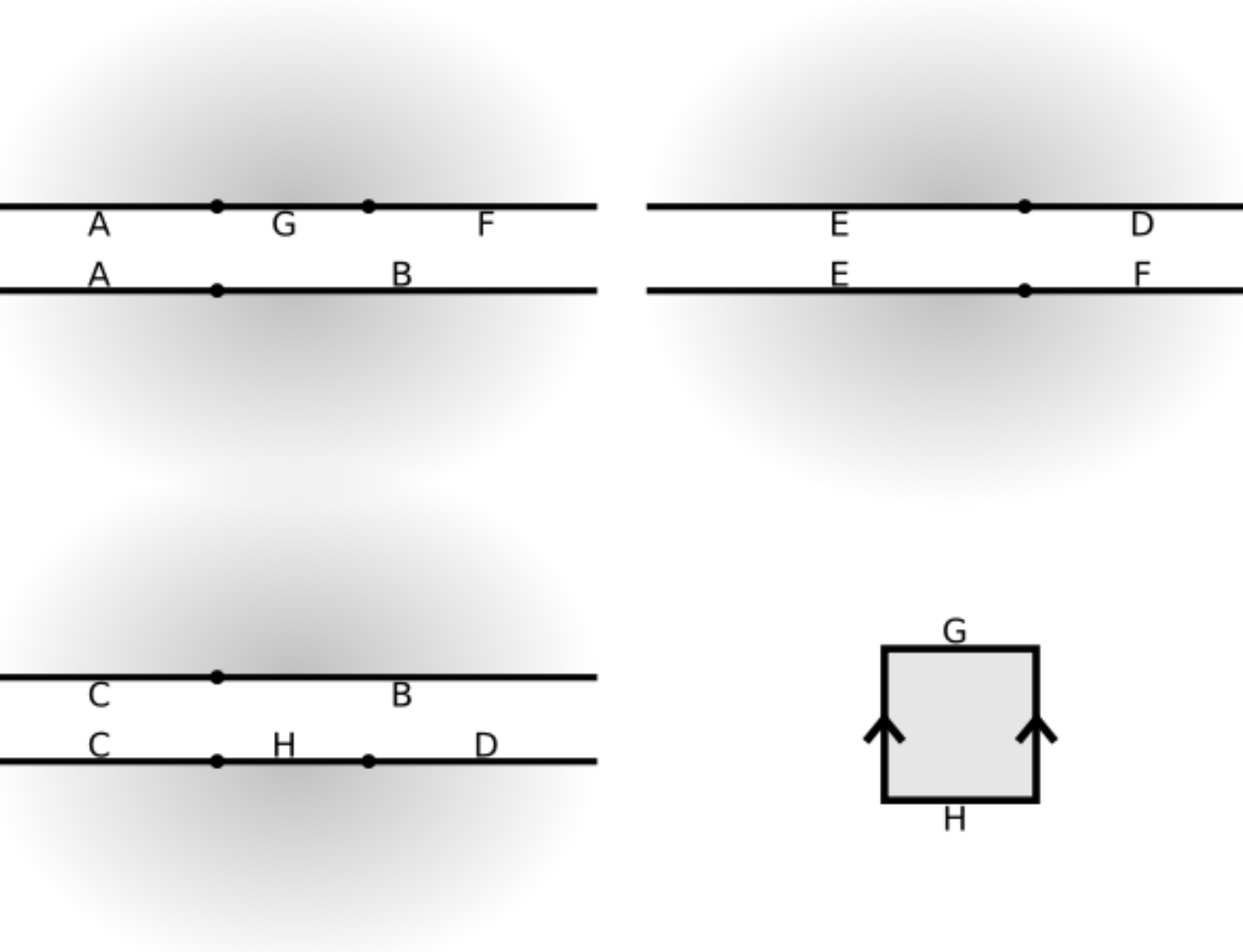}
\caption{
An infinite volume flat torus with zero of order $m=2$. (There is a pole at infinity which is not visible but which gives infinite volume.) The segments H and G have length one, and the remaining segments are bi-infinite.}
\label{F:InfiniteTorus}
\end{figure}

Now define $X_\infty$ to be $Y$ glued to the torus $T$ at a node placed at the zero of $Y$ and the pole of $T$.

\subsection{Convergence to $X_\infty$.} For this section our primary reference is \cite{Ba}, Sections 5.2-5.3, where Bainbridge discusses the topology on the {augmented Teichm\"uller space}, the Deligne-Mumford compactification and the bundle of stable Abelian differentials, recalling results of Abikoff \cite{Abi}, Bers \cite{Bers}, and others as well as results which are well known but not written in the literature.

The topology near $X_\infty$ may be described as follows. Let $U\subset X_\infty$ be any small neighborhood of the node in $X_\infty$. Consider the set $V_U$ of all nodal Riemann surfaces into which there is a holomorphic injection from $X_\infty-\overline{U}$. The $V_U$ form a neighborhood basis at $X_\infty$.

\begin{proof}[\eb{Proof of Proposition \ref{P:DM_limit}.}]
We will show that $X_n$ converges to $X_\infty$ by showing it is in the neighborhood $V_U$ above, where $U$ gets smaller as $n$ gets larger.

Indeed, suppose that on $X_n$ the size of the square handle is $\e_n$, where $\e_n\to 0$ as $n\to\infty$. Let $H_n\subset X_n$ be the set of points whose flat distance to the square handle (with respect to $\omega_n$) is less than $\sqrt{\e_n}$, and let $C_n=X_n\setminus \overline{H_n}$.

First we claim that there is a conformal map $\phi_n^H:H_n\to T$ for $n$ large enough. Indeed, this map simply rescales the flat metric by $\e_n^{-1}$ so that the handle now has size $1$, and then maps as a flat isometry into $T$ with the infinite volume flat metric given above. The image of $\phi_n$ is all points whose flat distance from the handle on $T$ is at most $\sqrt{\e_n^{-1}}$.

Next we claim that there is a conformal map $\phi_n^C:C_n\to Y$. Indeed, $C_n$ with the $\omega_n$ metric is flat isometric to the complement of a neighborhood of the zero in $(Y,\omega)$. This is because bubbling a handle into $(Y, \omega)$ is a \emph{local} surgery, and does not change the flat metric outside of a neighborhood of the zero of $\omega$.

Let $U_n=X_\infty\setminus \overline{\phi_n^C(C_n)\cup \phi_n^H(H_n)}.$ Then $U_n$ is a neighborhood of the node in $X_\infty$, and the intersection of all the $U_n$ is equal to the node. Define a map $\psi_n:X_\infty\setminus \overline{U_n} \to X_n$ by setting $\psi_n(z)=(\phi_n^C)^{-1}(z)$ for $z\in \phi_n^C(C_n)$, and $\psi_n(z)=(\phi_n^H)^{-1}(z)$ for $z\in \phi_n^H(H_n)$. This map is by definition a conformal injection $X_\infty\setminus \overline{U_n} \to X_n$, showing that $X_n\to X_\infty$.

The claim that $\omega_n$ converges to the Abelian differential $\omega$ supported on $Y\subset X_\infty$ is immediate from the description of the topology on the bundle of stable Abelian differentials given in \cite{Ba} (see p.60, under ``Long cylinders").

This proof applies equally well after applying a fixed $\newh \in SL(2,\bR)$ (the only difference is that the handles are no longer squares).
\end{proof}

\begin{proof}[\eb{Proof of Lemma \ref{L:glimit2}.}]
Let {$\newh \in SL(2,\bR)$}. In fact, the arguments in the previous three paragraphs above shows that there is a well defined nodal Riemann surface $X_\infty^{\newh }$. Here $X_\infty^{\newh }$ is the result of acting linearly by $\newh $ on the infinite volume flat structure on $X_\infty$ and taking the underlying Riemann surface structure. %This is the same as applying the construction immediately above to $g(t)(Y,\omega)$.
In particular, $X_n^{\newh }$ converges to $X_\infty^{\newh }$ in the Deligne-Mumford compactification. Here we are using the notation $\newh (X_n, \omega_n)=(X_n^{\newh }, \omega_n^{\newh })$ and $\newh (Y, \omega)=(Y^{\newh }, \omega^{\newh })$.

The bundle $H^1_\bR$ is trivial over augmented Teichm\"uller space. Set $P'$ to be the parallel transport of $P$ to $X_\infty$. (We are implicitly fixing markings here.)
It is now straightforward to see that $\newh (P_n)$ converges to  $P'$.  

Indeed, for $n$ sufficiently large depending on $h$, there is a natural identification of $H_1(X_n^\newh, \bR)$ with $H^1(Y^\newh, \bR)\oplus \bR^2$

Note that $H^1(X_\infty^{\newh }, \bR)=H_1(X_\infty^{\newh }, \bR)^*$, and the map $\iota^*$ is restriction to $H_1(Y^{\newh }, \bR)$. Using the above identification, we also have restriction maps $\iota^*:H^1(X_n^{\newh },\bR) \to H^1(Y^{\newh }, \bR)$, for $n$ sufficiently large.

Since the action of $\newh $ on the homology group $H_1(X_n,\bR)$ is trivial
%$$\newh _*: H_1(X_n, \bR) \to H_1(X_n^{\newh }, \bR),$$
 we see that $\iota^*\circ\newh =\newh \circ\iota^*.$ 

Thus $\iota^*(\newh  P_n)= \newh \iota^*(P_n)$ for all $n$ sufficiently large. Taking a limit we get that $\iota^*(P')=\newh \iota^*(P)$.
\end{proof}

\section{Open problems}\label{S:open}

Here we list three problems arising directly from our work.

\begin{prob}\label{p1}
Show that for all connected components of strata in genus at least 3 there is at least one surface whose only \good plane is the tautological plane.
\end{prob}

A solution to Problem \ref{p1} would help in the study of primitive but not algebraically primitive Teichm\"uller curves, as would a solution to the following.

\begin{prob}\label{p2}
Let $\cQ$ be a connected component of a stratum of quadratic differentials, and let $\cM$ be the corresponding affine invariant submanifold of translation surfaces that are holonomy double covers of half-translation surfaces in $\cQ$. Show that except in certain small genus cases (such as Prym loci) there is a translation surface in $\cM$ whose only \good plane is the tautological plane.
\end{prob}

Problems \ref{p1} and  \ref{p2} are closely related to the ongoing study of the Kontsevich-Zorich cocycle.

\begin{prob}\label{p3}
In each connected component of each stratum in genus $g$ at least 3, give an explicit open set of translation surfaces without $g$ orthogonal \good planes.
\end{prob}

The simplest case is probably the hyperelliptic connected component of $\cH(4)$. Even in that case, we do not know how to do Problem \ref{p3}, however we do not feel that it is completely hopeless. For example, a study of the second derivative of the period matrix might be helpful.

\bibliography{mybib}{}
\bibliographystyle{amsalpha}

\end{document}